\documentclass[11pt]{amsart}
\usepackage[latin1]{inputenc}
\usepackage{amsmath,amsthm,amsfonts,amssymb,amscd}
\usepackage{verbatim}
\usepackage{multicol}
\usepackage{tabularx}
\usepackage[usenames]{color}
\usepackage{graphicx}
\usepackage[all]{xy}

\theoremstyle{definition}
\newtheorem{thm}{Theorem}[section]
\newtheorem{prop}[thm]{Proposition}
\newtheorem{cor}[thm]{Corollary}
\newtheorem{con}[thm]{Conjecture}
\newtheorem{lem}[thm]{Lemma}
\newtheorem{defn}[thm]{Definition}

\newtheorem*{rem}{Remark}

\numberwithin{equation}{section}

\def\blb#1{\text{$\mathbb{#1}$}}
\def\cal#1{\text{$\mathcal{#1}$}}

\def\ord#1^#2{#1$^{\text{#2}}$}

\def\lie#1{\mathfrak{#1}}
\def\tlie#1{\tilde{\mathfrak{#1}}}
\def\hlie#1{\hat{\mathfrak{#1}}}

\def\uqr#1^#2{\text{$U_q^{#2}(\lie #1)$}}

\def\uqhr#1^#2{\text{$U_q^{#2}(\hlie #1)$}}
\def\us#1^#2{\text{$U_{\xi}^{#2}(\lie #1)$}}
\def\ush#1^#2{\text{$U_{\xi}^{#2}(\hlie #1)$}}
\def\dus#1^#2{\text{$\dot{U}_{\xi}^{#2}(\lie #1)$}}
\def\dush#1^#2{\text{$\dot{U}_{\xi}^{#2}(\hlie #1)$}}
\def\gb#1{{\mbox{\boldmath $#1$}}}
\def\gbr#1{{\mbox{\boldmath ${\rm #1}$}}}
\def\wtl{{\rm wt}_\ell}
\def\wt{{\rm wt}}
\def\het{{\rm ht}}
\def\supp{{\rm supp}}
\def\gr{{\rm gr}}

\def\ch{{\rm char}}

\def\opl_#1^#2{\text{\scriptsize$\bigoplus\limits_{\text{\normalsize$#1$}}^{\text{\normalsize$#2$}}$}}
\def\otm_#1^#2{\text{\scriptsize$\bigotimes\limits_{\text{\footnotesize$#1$}}^{\text{\footnotesize$#2$}}$}}
\def\wcal#1{{\mbox{$\widetilde{\cal #1}$}}}

\def\bgb#1{{\mbox{$\overline{\gb #1}$}}}

\def\tqbinom#1#2{\text{$\left[\begin{smallmatrix} #1\\#2\end{smallmatrix}\right]$}}

\renewcommand{\thefootnote}

\begin{document}

\title[Graded limits of minimal affinizations and Beyond]{Graded limits of minimal affinizations and Beyond:\\ the multiplicty free case for type $E_6$}
\author[Adriano Moura and Fernanda Pereira]{Adriano Moura and Fernanda Pereira}
\thanks{}
\address{Departamento de Matemática, Universidade Estadual de Campinas, Campinas - SP - Brazil, 13083-859.}
\email{aamoura@ime.unicamp.br and fernandapereira@ime.unicamp.br}

\maketitle
\centerline{\small{
\begin{minipage}{350pt}
{\bf Abstract:} We obtain a graded character formula for certain graded modules for the current algebra over a simple Lie algebra of type $E_6$. For certain values of their highest weight, these modules were conjectured to be isomorphic to the classical limit of the corresponding minimal affinizations of the associated quantum group. We prove that this is the case under further restrictions on the highest weight. Under another set of conditions on the highest weight, Chari and Greenstein have recently proved that they are projective objects of a full subcategory of the category of graded modules for the current algebra. Our formula applies to all of these projective modules.
\end{minipage}}}

\setcounter{section}{0}

\section*{Introduction}

The problem of determining the structure of the minimal affinizations of quantum groups  is one of the most studied problems in the finite-dimensional representation theory of quantum affine algebras in recent years (see \cite{chhe:beyond} for a recent survey with a comprehensive list of references). In particular, determining the character of such representations when regarded as modules for the quantum group $U_q(\lie g)$ over the underlying semisimple Lie algebra $\lie g$ is of special interest.
Determining the character is theoretically equivalent to determining the multiplicity of the irreducible constituents of these representations when regarded as $U_q(\lie g)$-modules. In practice, computing the multiplicities out of a given character is a laborious task which can be performed algorithmically.

One of the methods which have been used to approach this problem is that of considering the classical limit of the given module and regard it as a representation for the current algebra $\lie g[t]=\lie g\otimes\mathbb C[t]$. This approach was first considered in \cite{cha:fer,chakle} and it was then further developed in \cite{cm:kr,cm:krg,mou:reslim}. In this paper, we apply this method for $\lie g$ of type $E_6$ and obtain a formula for the multiplicities of the irreducible constituents of the graded pieces of these modules assuming certain conditions on the highest weight. Our formula actually holds for a larger class of $\lie g[t]$-modules. Namely, given a dominant integral weight $\lambda$ of $\lie g$, the first author defined in \cite{mou:reslim} a $\lie g[t]$-module denoted by $M(\lambda)$. The definition is by generator and relations which naturally generalize the relations of the classical limits of Kirillov-Reshetikhin modules obtained in \cite{cha:fer}. It was conjectured in \cite{mou:reslim} that $M(\lambda)$ is isomorphic to  the classical limit of the minimal affinizations of the irreducible $U_q(\lie g)$-module of highest-weight $\lambda$ provided that there exists a unique equivalence class of minimal affinizations associated to $\lambda$. Our main results are a formula for the multiplicities of the irreducible constituents of the graded pieces of the modules $M(\lambda)$ and the proof of the conjecture of \cite{mou:reslim} assuming certain conditions on $\lambda$. To explain these conditions, let us label the nodes of the Dynkin diagram of $\lie g$ as follows.
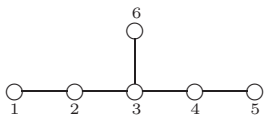
\begin{displaymath}
\setlength{\unitlength}{.2cm}
\begin{picture}(18,6)(0,0)
\put(0,0.5){\circle{1}} \put(-0.3,-1){$\scriptscriptstyle{1}$}
\put(0.5,0.5){\line(1,0){3}} \put(4,0.5){\circle{1}}
\put(3.7,-1){$\scriptscriptstyle{2}$} \put(4.5,0.5){\line(1,0){3}}
\put(8,0.5){\circle{1}} \put(7.8,-1){$\scriptscriptstyle{3}$}
\put(8.5,0.5){\line(1,0){3}} \put(12,0.5){\circle{1}}
\put(11.7,-1){$\scriptscriptstyle{4}$} \put(12.5,0.5){\line(1,0){3}}
\put(16,0.5){\circle{1}} \put(15.7,-1){$\scriptscriptstyle{5}$}
\put(8,1){\line(0,1){3}} \put(8,4.5){\circle{1}}
\put(7.8,5.4){$\scriptscriptstyle{6}$}
\end{picture}
\end{displaymath}

Let $I=\{1,2,\dots,6\}$ and identify it with the set of nodes of the Dynkin diagram of $\lie g$ following the above labeling. For an integral weight $\mu$, the support of $\mu$ is the subset of $I$ consisting of labels such that the value of $\mu$ on the corresponding co-root is nonzero. The connected closure of the support is the minimal connected subdiagram of the Dynkin diagram of $\lie g$ containing the nodes in the support of $\mu$. We mostly focus our study on the modules $M(\lambda)$ with $\lambda$ not supported in the trivalent node and prove that the character formula \eqref{e:chM} below holds for all $\lambda$ with support contained in one of the following subsets of $I$: $\{1,2,5,6\},\{1,4,5,6\},\{2,4,6\}$. Following the conjecture of \cite{mou:reslim}, we conjecture that \eqref{e:chM} holds for all $\lambda$ not supported in the trivalent node and prove in such generality that \eqref{e:chM} gives an upper bound for the multiplicities of the $\lie g$-irreducible constituents of the graded pieces of $M(\lambda)$ (see \eqref{e:ub}). In particular, it follows from \eqref{e:multfree} that all irreducible constituents are multiplicity free (even if the grading is not taken into account). As a byproduct of the proof of \eqref{e:chM}, we obtain a realization of $M(\lambda)$ as a submodule of the tensor product of the classical limits of certain Kirillov-Reshetikhin modules (Theorem \ref{t:main}(a)), thus establishing  part of the conjecture of \cite{mou:reslim} for such $\lambda$.

Keeping the above conditions on $\lambda$ and further assuming that the connected closure of the support of $\lambda$ is of type $A$, we prove that $M(\lambda)$ is isomorphic to the classical limit of the corresponding minimal affinizations when regarded as $\lie g[t]$-modules (Theorem \ref{t:main}(b)). This establishes the other part of the conjecture of \cite{mou:reslim} for these values of $\lambda$. In particular, \eqref{e:chM} gives the multiplicities of the irreducible constituents of the minimal affinizations when the support of $\lambda$ is contained in one of the following subsets of $I$: $\{1,2,5\},\{1,4,5\},\{1,2,6\},\{4,5,6\},\{2,4\}$. Moreover, we also prove that, if \eqref{e:chM} indeed holds for any $\lambda$ not supported in the trivalent node as conjectured, then we can include $\{1,2,4,5\}$ in this list.
Dropping all the assumptions on $\lambda$ except that the connected closure of its support is of type $A$,  we prove that the classical limit of the corresponding minimal affinizations are quotients of $M(\lambda)$ (Proposition \ref{p:LqM}). This is a further step towards the proof of the conjecture of \cite{mou:reslim} in general. However, the graded character formula for the Kirillov-Reshetikhin modules associated to the trivalent node given in \cite{jap:rem} implies that, if $\lambda$ is supported on that node, then these modules are not multiplicity free. We remark that, in \cite{nak:qvtqchar},  Nakajima developed an algorithm for computing the $t$-analogue of the q-character of any finite-dimensional irreducible representation of the quantum affine algebra associated to any simply laced simple Lie algebra $\lie g$. In particular, without any assumption on $\lambda$, the graded character of the classical limits of the minimal affinizations associated to $\lambda$ can be computed using this algorithm.  Theoretically, one can then compute the multiplicities from the character as mentioned in the first paragraph of this introduction. On the other hand, with the above assumptions on $\lambda$, formula \eqref{e:chM} gives these multiplicities directly.

Let us explain the reasons behind the several aforementioned restrictions on $\lambda$. First we recall that, for simply laced $\lie g$, there exists a unique equivalence class of minimal affinizations associated to $\lambda$ if and only if the connected closure of its support is of type $A$. Let $\theta$ be the highest root of $\lie g$ and, given $i\in I$, let $\epsilon_i(\theta)$ be its coordinate in the basis of simple roots. Given a positive integer $r$, let $\lie g[t:r]$ be the quotient of $\lie g[t]$ by the ideal $\lie g\otimes t^r\mathbb C[t]$. It turns out that $M(\lambda)$ factors to a module for $\lie g[t:r]$ where $r$ is the maximum of $\epsilon_i(\theta)$ for $i$ running on the support of $\lambda$. In particular, if $\lie g$ is of type $E_6$, $M(\lambda)$ can be regarded as a module for $\lie g[t:3]$. Moreover,  if $\lambda$ is not supported on the trivalent node, then $M(\lambda)$ factors to a module for $\lie g[t:2]$. The category $\cal G_2$ of graded $\lie g[t:2]$-modules with finite-dimensional graded pieces has been recently studied in \cite{cg:kos,cg:proj} by exploring its interplay with the theory of Koszul algebras and quiver representations. The literature on the representation theory of $\lie g[t:r]$ for $r>2$ is more limited and results such as the ones from \cite{cg:kos,cg:proj} are yet to be established. Thus, we focus on the case that $M(\lambda)$ factors to a $\lie g[t:2]$-module which, for type $E_6$, is equivalent to assuming that $\lambda$ is not supported on the trivalent node (as mentioned above, this is also the necessary and sufficient condition for the modules $M(\lambda)$ to be multiplicity free). It follows from \cite[Theorem 1]{cg:proj} that, if $\lambda$ satisfies certain conditions, then $M(\lambda)$ is a projective object of a full subcategory of $\cal G_2$ naturally attached to $\lambda$. Moreover, \cite[Theorem 2]{cg:proj} gives a graded character formula for $M(\lambda)$ provided $\lambda$ satisfies the conditions of \cite[Theorem 1]{cg:proj}. We remark that \cite[Theorem 2]{cg:proj} expresses the graded character of $M(\lambda)$ in terms of an alternating sum of the graded characters of $M(\mu)$ with $\mu$ strictly smaller than $\lambda$ with respect to the usual partial order on the weight lattice of $\lie g$. Hence, the formula of \cite[Theorem 2]{cg:proj} is of recursive nature. For $\lie g$ of type $E_6$, we prove that the conditions on $\lambda$ required on \cite[Theorem 1]{cg:proj} is equivalent to requiring that the support of $\lambda$ be contained in one of the following subsets of $I$: $\{1,2,5,6\},\{1,4,5,6\}$. Therefore, \eqref{e:chM} holds beyond the cases covered by \cite[Theorem 2]{cg:proj}. This latter list of subsets of $I$ also hints that it should be expected that  when the support of $\lambda$ contains $\{2,4\}$ the situation should be more complicated than otherwise. Indeed, the proof of \eqref{e:chM} for this case is significantly more technically involved than for the others.

The paper is organized as follows. In Section \ref{s:algs}, we review the basic notation on simply laced simple Lie algebras and the associated loop algebras, current algebras, quantum groups, and quantum affine algebras. In Section \ref{s:fdimreps}, we review the relevant facts on the finite-dimensional representation theory of these algebras. After reviewing the classification of minimal affinizations in Subsection \ref{ss:min}, the main results (Theorem \ref{t:main}, Proposition \ref{p:LqM}, the multiplicity free property \eqref{e:ub}, and the character formula \eqref{e:chM}) are stated in Subsection \ref{ss:limits}. The relation of our results with those of \cite{cg:proj} is explained in Subsection \ref{ss:proj}. The proofs are given in Section \ref{s:proof}.

\vskip15pt
\noindent{\bf Acknowledgements:}  The work of the first author was partially supported by CNPq. The M.Sc. studies of the second author, during which part of this work was done, were supported by FAPESP.

\section{Quantum and classical loop algebras}\label{s:algs}

Throughout the paper, let $\mathbb C, \mathbb R,\mathbb Z,\mathbb Z_{\ge m}$ denote the sets of complex numbers, reals, integers,  and integers bigger or equal $m$, respectively. Given a ring $\mathbb A$, the underlying multiplicative group of units is denoted by $\mathbb A^\times$. The dual of a vector space $V$ is denoted by $V^*$. The symbol $\cong$ means ``isomorphic to''. The cardinality of a set $S$ will be denoted by $|S|$.

\subsection{Classical algebras}\label{ss:clalg}

Let $I=\{1,\dots,n\}$ be the set of vertices of a finite-type simply laced Dynkin diagram and let $\lie g$ be the associated semisimple Lie algebra over $\mathbb C$ with a fixed Cartan subalgebra $\lie h$. Fix a set of positive roots $R^+$ and let
$$\lie n^\pm = \opl_{\alpha\in R^+}^{} \lie g_{\pm\alpha} \quad\text{where}\quad \lie g_{\pm\alpha} = \{x\in\lie g: [h,x]=\pm\alpha(h)x, \ \forall \ h\in\lie h\}.$$
The simple roots will be denoted by $\alpha_i$ and the fundamental weights by $\omega_i$, $i\in I$. $Q,P,Q^+,P^+$ will denote the root and weight lattices with corresponding positive cones, respectively.
Let also $h_i\in\lie h$, be the co-root associated to $\alpha_i, i\in I$. We equip $\lie h^*$ with the partial order $\lambda\le \mu$ iff $\mu-\lambda\in Q^+$. Let $C = (c_{ij})_{i,j\in I}$ be the Cartan matrix of $\lie g$, i.e., $c_{ij}=\alpha_j(h_i)$. The Weyl group is denoted by $\cal W$.

The subalgebras $\lie g_{\pm\alpha}, \alpha\in R^+$, are one-dimensional and $[\lie g_{\pm\alpha},\lie g_{\pm\beta}]=\lie g_{\pm\alpha\pm\beta}$ for every $\alpha,\beta\in R^+$. We denote by $x^\pm_\alpha$ any generator of $\lie g_{\pm\alpha}$ and, in case $\alpha=\alpha_i$ for some $i\in I$, we may also use the notation $x_i^\pm$ in place of $x_{\alpha_i}^\pm$. In particular, if $\alpha+\beta\in R^+$, $[x^\pm_\alpha,x^\pm_\beta]$ is a nonzero generator of $\lie g_{\pm\alpha\pm\beta}$ and we simply write  $[x^\pm_\alpha,x^\pm_\beta]=x_{\alpha+\beta}^\pm$. For each subset $J$ of $I$ let $\lie g_J$ be the Lie subalgebra of $\lie g$ generated by $x_{\alpha_j}^\pm, j\in J$, and define $\lie n^\pm_J, \lie h_J$ in the obvious way. Let also $Q_J$ be the subgroup of $Q$ generated by $\alpha_j, j\in J$, and $R^+_J=R^+\cap Q_J$. Given $\lambda\in P$, let $\lambda_J$ be the restriction of $\lambda$ to $\lie h_J^*$ and $\lambda^J\in P$ be such that $\lambda^J(h_j)=\lambda(h_j)$ if $j\in J$ and $\lambda^J(h_j)=0$ otherwise. By abuse of language, we will refer to any subset $J$ of $I$ as a subdiagram of the Dynkin diagram of $\lie g$. The support of $\mu\in P$ is defined to be the subdiagram ${\rm supp}(\mu)\subseteq I$ given by ${\rm supp}(\mu)=\{i\in I:\mu(h_i)\ne 0\}$. Let also $\overline{\rm supp}(\mu)$ be the minimal connected subdiagram of $I$ containing ${\rm supp}(\mu)$.

If  $\lie a$ is a Lie algebra over $\mathbb C$, define its loop algebra to be $\tlie a=\lie a\otimes_{\mathbb C}  \mathbb C[t,t^{-1}]$ with bracket given by $[x \otimes t^r,y \otimes t^s]=[x,y] \otimes t^{r+s}$. Clearly $\lie a\otimes 1$ is a subalgebra of $\tlie a$ isomorphic to $\lie a$ and, by abuse of notation, we will continue denoting its elements by $x$ instead of $x\otimes 1$. We also consider the current algebra $\lie a[t]$ which is the subalgebra of $\tlie a$ given by $\lie a[t]=\lie a\otimes \mathbb C[t]$. Then $\tlie g = \tlie n^-\oplus \tlie h\oplus \tlie n^+$ and $\tlie h$ is an abelian subalgebra and similarly for $\lie g[t]$. The elements $x_\alpha^\pm\otimes t^r, x_i^\pm\otimes t^r$, and $h_i\otimes t^r$ will be denoted by $x_{\alpha,r}^\pm, x_{i,r}^\pm$, and $h_{i,r}$, respectively. Also,  Diagram subalgebras $\tlie g_J$ are defined in the obvious way.

Let $U(\lie a)$ denote the universal enveloping algebra of a Lie algebra $\lie a$. Then $U(\lie a)$ is a subalgebra of $U(\tlie a)$.
Given $a\in\mathbb C$, let $\tau_a$ be the Lie algebra automorphism of $\lie a[t]$ defined by $\tau_a(x\otimes f(t))=x\otimes f(t-a)$ for every $x\in\lie a$ and every $f(t)\in\mathbb C[t]$. If $a\ne 0$, let ${\rm ev}_a:\tlie a\to\lie a$ be the evaluation map $x\otimes f(t)\mapsto f(a)x$. We also denote by $\tau_a$ and ${\rm ev}_a$ the induced maps $U(\lie a[t])\to U(\lie a[t])$ and $U(\tlie a)\to U(\lie a)$, respectively. Given a nonzero $x\in \lie a$ we shall denote by $U(x)$ the universal enveloping algebra of the one-dimensional subalgebra generated by $x$ regarded as a subalgebra of $U(\lie a)$.

For each $i\in I$ and  $r\in\mathbb Z$, define elements $\Lambda_{i,r}\in U(\tlie h)$ by the following equality of formal power series in the variable $u$:
\begin{equation}\label{e:Lambdadef}
\sum_{r=0}^\infty \Lambda_{i,\pm r} u^r = \exp\left( - \sum_{s=1}^\infty \frac{h_{\alpha_i,\pm s}}{s} u^s\right).
\end{equation}

\subsection{Quantum algebras}

Let $\mathbb C(q)$ be the ring of rational functions on an indeterminate $q$ and $\mathbb A=\mathbb C[q,q^{-1}]$. Set
\begin{equation*}
[m]=\frac{q^m -q^{-m}}{q -q^{-1}},\ \ \ \ [m]!
=[m][m-1]\ldots [2][1],\ \ \ \ \tqbinom{m}{r} = \frac{[m]!}{[r]![m-r]!},
\end{equation*}
for $r,m\in\mathbb Z_{\ge 0}$, $m\ge r$. Notice that $[m],\tqbinom{m}{r}\in\mathbb A$.

The quantum loop algebra  $U_q(\tlie g)$ of $\lie g$ is  the associative $\mathbb C(q)$-algebra with
generators $x_{i,r}^{{}\pm{}}$ ($i\in I$, $r\in\blb Z$), $k_i^{{}\pm
1}$ ($i\in I$), $h_{i,r}$ ($i\in I$, $r\in \blb Z\backslash\{0\}$)
and the following defining relations:
\begin{align*}
k_ik_i^{-1} = k_i^{-1}k_i& =1, \ \  k_ik_j =k_jk_i,\\
k_ih_{j,r}& =h_{j,r}k_i,\\
k_ix_{j,r}^\pm k_i^{-1} &= q^{{}\pm c_{ij}}x_{j,r}^{{}\pm{}},\ \ \\
[h_{i,r},h_{j,s}]=0,\; \; & [h_{i,r} , x_{j,s}^{{}\pm{}}] = \pm\frac1r[rc_{ij}]x_{j,r+s}^{{}\pm{}},\\
x_{i,r+1}^{{}\pm{}}x_{j,s}^{{}\pm{}} -q^{{}\pm c_{ij}}x_{j,s}^{{}\pm{}}x_{i,r+1}^{{}\pm{}} &
=q^{{}\pm     c_{ij}}x_{i,r}^{{}\pm{}}x_{j,s+1}^{{}\pm{}} -x_{j,s+1}^{{}\pm{}}x_{i,r}^{{}\pm{}},\\
[x_{i,r}^+ , x_{j,s}^-]=\delta_{i,j} & \frac{ \psi_{i,r+s}^+ - \psi_{i,r+s}^-}{q - q^{-1}},\\
\sum_{\sigma\in S_m}\sum_{k=0}^m(-1)^k
\tqbinom{m}{k}x_{i, r_{\sigma(1)}}^{{}\pm{}}\ldots x_{i,r_{\sigma(k)}}^{{}\pm{}} &
x_{j,s}^{{}\pm{}} x_{i, r_{\sigma(k+1)}}^{{}\pm{}}\ldots x_{i,r_{\sigma(m)}}^{{}\pm{}} =0,\ \ \text{if $i\ne j$},
\end{align*}
for all sequences of integers $r_1,\ldots, r_m$, where $m
=1-c_{ij}$, $S_m$ is the symmetric group on $m$ letters, and
the $\psi_{i,r}^{{}\pm{}}$ are determined by equating powers of
$u$ in the formal power series
$$\Psi_i^\pm(u) = \sum_{r=0}^{\infty}\psi_{i,\pm r}^{\pm}u^r = k_i^{\pm 1} \exp\left(\pm(q-q^{-1})\sum_{s=1}^{\infty}h_{i,\pm s} u^s\right).$$

Denote by $U_q(\tlie n^\pm), U_q(\tlie h)$  the subalgebras of $U_q(\tlie g)$ generated by $\{x_{i,r}^\pm\}, \{k_i^{\pm1}, h_{i,s}\}$, respectively. Let  $U_q(\lie g)$  be  the subalgebra generated by $x_i^\pm:=x_{i,0}^\pm, k_i^{\pm 1}, i\in I,$ and define $U_q(\lie n^\pm), U_q(\lie h)$ in the obvious way.  $U_q(\lie g)$ is a subalgebra of $U_q(\tlie g)$ and multiplication establishes isomorphisms of $\mathbb C(q)$-vectors spaces:
$$U_q(\lie g) \cong U_q(\lie n^-) \otimes U_q(\lie h) \otimes U_q(\lie n^+) \qquad\text{and}\qquad U_q(\tlie g) \cong U_q(\tlie n^-) \otimes U_q(\tlie h) \otimes U_q(\tlie n^+).$$

Let $J\subseteq I$ and consider the subalgebra $U_q(\tlie g_J)$ generated by $k_j^{\pm 1}, h_{j,r}, x^\pm_{j,s}$ for all $j\in J, r,s\in \blb Z, r\ne 0$. If $J=\{j\}$, the algebra $U_q(\tlie g_j):=U_q(\tlie g_J)$ is isomorphic to $U_{q}(\tlie{sl}_2)$. Similarly we define the subalgebra $U_q(\lie g_J)$, etc.

For $i\in I, r\in \mathbb Z, k\in\mathbb Z_{\ge 0}$, define $(x_{i,r}^\pm)^{(k)} = \frac{(x_{i,r}^\pm)^k}{[k]!}$. Define also elements $\Lambda_{i,r}, i\in I, r\in\mathbb Z$ by
\begin{equation}\label{e:Lambdad}
\sum_{r=0}^\infty \Lambda_{i,\pm r} u^{r}=
\exp\left(-\sum_{s=1}^\infty\frac{h_{i,\pm s}}{[s]}u^s\right).
\end{equation}
Although we are denoting the elements $x_{i,r}^\pm, h_{i,r}$, and $\Lambda_{i,r}$ above by the same symbol as their classical counterparts, this will not create confusion as it will be clear from the context.

Let $U_{\mathbb A}(\tlie g)$ be the $\mathbb A$-subalgebra of $U_q(\tlie g)$ generated by the elements $(x_{i,r}^\pm)^{(k)}, k_i^{\pm 1}$ for $i\in I,r\in\mathbb Z$, and $k\in\mathbb Z_{\ge 0}$. Define $U_{\mathbb A}(\lie g)$ similarly and notice that $U_{\mathbb A}(\lie g)=U_{\mathbb A}(\tlie g)\cap U_q(\lie g)$. Henceforth $\lie a$ will denote a Lie algebra of the following set: $\lie g, \lie n^\pm, \lie h, \tlie g, \tlie n^\pm, \tlie h$. For the proof of the next proposition see \cite[Lemma 2.1]{cha:fer} and the locally cited references.

\begin{prop}
The  canonical map $\mathbb C(q)\otimes_{\mathbb A} U_{\mathbb A}(\lie a)\to U_q(\lie a)$ is an isomorphism.\hfill\qedsymbol
\end{prop}

Regard $\mathbb C$ as an $\mathbb A$-module by letting $q$ act as $1$ and set
\begin{equation}
\overline{U_q(\lie a)} = \mathbb C\otimes_\mathbb A U_\mathbb A(\lie a).
\end{equation}
Denote by $\bar x$ the image of $x\in U_\mathbb A(\tlie g)$ in $\overline{U_q(\tlie g)}$. For a proof of the next proposition see \cite[Proposition 9.2.3]{cp:book} and the locally cited references.

\begin{prop}\label{p:cluq}
$U(\tlie g)$ is isomorphic to the quotient of $\overline{U_q(\tlie g)}$ by the ideal generated by ${\overline k_i}-1$. In particular, the category of $\overline{U_q(\tlie g)}$-modules on which $k_i$ act as the identity operator for all $i\in I$ is equivalent to the category of all $\tlie g$-modules.\hfill\qedsymbol
\end{prop}

The algebra $U_q(\tlie g)$ is a Hopf algebra and induces a Hopf algebra structure (over $\mathbb A$) on $U_\mathbb A(\tlie g)$. Moreover, the induced Hopf algebra structure on $U(\tlie g)$ coincides with the usual one (see \cite{cp:book,lus:book}). On $U_q(\lie g)$ we have
\begin{equation}\label{e:comultuqg}
\Delta (x_{i}^+) = x_{i}^+\otimes 1+ k_i\otimes x_{i}^+, \qquad \Delta (x_{i}^-) = x_{i}^-\otimes k_i^{-1}  + 1\otimes x_{i}^-, \qquad \Delta(k_i) = k_i\otimes k_i
\end{equation}
for all $i\in I$. The next lemma is easily established (cf. \cite[Lemma 1.5]{mou:reslim}).

\begin{lem}\label{l:comcom}
Suppose $x=[x_{i_1}^-,[x_{i_2}^-,\cdots[x_{i_{l-1}}^-,x_{i_l}^-]\cdots]]$. Then $x\in U_\mathbb A(\lie n^-)$ and
$$\Delta(x)\in x\otimes (\prod_{j=1}^l k_{i_j}^{-1}) + 1\otimes x + f(q)y$$
for some $y\in U_\mathbb A(\lie g)\otimes U_\mathbb A(\lie g)$ and some $f(q)\in\mathbb A$ such that $f(1)=0$.\hfill\qedsymbol
\end{lem}

An expression for the comultiplication $\Delta$ of $U_q(\tlie g)$ in terms of the generators $x^\pm_{i,r}, h_{i,r}, k_i^{\pm1}$  is not known. The following partial information will suffice for our purposes (see \cite[Lemma 1.6]{mou:reslim} and the locally cited references).

\begin{lem}\label{l:comult}
$\Delta(x_{i,1}^-)=x_{i,1}^-\otimes k_i + 1\otimes x_{i,1}^- + x$ for some $x\in U_\mathbb A(\lie g)\otimes U_\mathbb A(\lie g)$ such that $\bar x=0$.\hfill\qedsymbol
\end{lem}

\subsection{The $\ell$-weight lattice}\label{ss:llattice} Given a field $\mathbb F$ consider the multiplicative group $\cal P_\mathbb F$ of $n$-tuples of rational functions $\gb\mu = (\gb\mu_1(u),\cdots, \gb\mu_n(u))$ with values in $\mathbb F$  such that $\gb\mu_i(0)=1$ for all $i\in I$. We shall often think of $\gb\mu_i(u)$ as a formal power series in $u$ with coefficients in $\mathbb F$. Given $a\in\mathbb F^\times$ and $i\in I$, let $\gb\omega_{i,a}$ be defined by
$$(\gb\omega_{i,a})_j(u) = 1-\delta_{i,j}au.$$
Clearly, if $\mathbb F$ is algebraically closed, $\cal P_\mathbb F$ is the free abelian group generated by these elements which are called fundamental $\ell$-weights. It is also convenient to introduce elements $\gb\omega_{\lambda,a}, \lambda\in P,a\in\mathbb F$, defined by
\begin{equation}
\gb\omega_{\lambda,a} = \prod_{i\in I}(\gb\omega_{i,a})^{\lambda(h_i)}.
\end{equation}
If $\mathbb F$ is algebraically closed, introduce the group homomorphism (weight map) $\wt:\cal P_\mathbb F \to P$ by setting $\wt(\gb\omega_{i,a})=\omega_i$. Otherwise, let $\mathbb K$ be an algebraically closed extension of $\mathbb F$ so that $\cal P_\mathbb F$ can be regarded as a subgroup of $\cal P_\mathbb K$ and define the weight map on $\cal P_\mathbb F$ by restricting the one on $\cal P_\mathbb K$.

Define the $\ell$-weight lattice of $U_q(\tlie g)$ to be $\cal P_q:=\cal P_{\mathbb C(q)}$. The submonoid $\cal P_q^+$ of $\cal P_q$ consisting of $n$-tuples of polynomials is called the set of dominant $\ell$-weights of $U_q(\tlie g)$.
Given $\gb\lambda\in\cal P_q^+$ with $\gb\lambda_i(u) = \prod_j (1-a_{i,j}u)$, where $a_{i,j}$ belongs to some algebraic closure of $\mathbb C(q)$, let $\gb\lambda^-\in\cal P_q^+$ be defined by $\gb\lambda^-_i(u) = \prod_j (1-a_{i,j}^{-1}u)$. We will also use the notation $\gb\lambda^+ = \gb\lambda$.  Given $\gb\nu\in\cal P_q$, say $\gb\nu = \gb\lambda\gb\mu^{-1}$ with $\gb\lambda,\gb\mu\in\cal P_q^+$, define a $\mathbb C(q)$-algebra homomorphism $\gb\Psi_{\gb\nu}:U_q(\tlie h)\to \mathbb C(q)$ by  setting $\gb\Psi_{\gb\nu}(k_i^{\pm 1}) = q_i^{\pm \wt(\gb\nu)(h_i)}$ and
\begin{equation}\label{e:PsiLambda}
\sum_{r\ge 0} \gb\Psi_{\gb\nu}(\Lambda_{i,\pm r}) u^r = \frac{(\gb\lambda^{\pm})_i(u)}{(\gb\mu^{\pm})_i(u)}.
\end{equation}
One easily checks that the map $\gb\Psi:\cal P_q\to (U_q(\tlie h))^*$ given by $\gb\nu\mapsto \gb\Psi_{\gb\nu}$ is injective.
Define the $\ell$-weight lattice $\cal P$ of $\tlie g$ to be the subgroup of $\cal P_q$ generated by $\gb\omega_{i,a}$ for all $i\in I$ and all $a\in\mathbb C^\times$ or, equivalently, $\cal P = \cal P_\mathbb C$.
Set also $\cal P^+=\cal P\cap\cal P_q^+$. From now on we will identify $\cal P_q$  with its image in $(U_q(\tlie h))^*$  under $\gb\Psi$. Similarly, $\cal P$ will be identified with a subset of $U(\tlie h)^*$ via the homomorphism $\gb\Psi_{\gb\nu}:U(\tlie h)\to \mathbb C$ determined by \eqref{e:PsiLambda} and $\gb\Psi_{\gb\nu}(h_i) = \wt(\gb\nu)(h_i)$.

It will be convenient to introduce the following notation. Given $i\in I, a\in\mathbb C(q)^\times, r\in\mathbb Z_{\ge 0}$, define
\begin{equation}\label{e:krhlw}
\gb\omega_{i,a,r} = \prod_{j=0}^{r-1} \gb\omega_{i,aq^{r-1-2j}}.
\end{equation}

If $J\subseteq I$ and $\gb\lambda\in\cal P_q$, let $\gb\lambda_J$ be the associated $J$-tuple of rational functions. Notice that, if $\gb\lambda_j(u)\in\mathbb C(q_j)(u)$ for all $j\in J$, $\gb\lambda_J$ can be regarded as an element of the $\ell$-weight lattice of $U_q(\tlie g_J)$. Let also $\gb\lambda^J\in\cal P_q$ be such that $(\gb\lambda^J)_j(u)=\gb\lambda_j(u)$ for every $j\in J$ and $(\gb\lambda^J)_j(u)=1$ otherwise.

Given $i\in I$ and $a\in\mathbb C(q)^\times$, define the simple $\ell$-root $\gb\alpha_{i,a}$ by
\begin{equation}
\gb\alpha_{i,a} = \gb\omega_{i,aq,2}\prod_{j\ne i} \gb\omega_{j,aq,-c_{j,i}}^{-1}.
\end{equation}
The subgroup of $\cal P_q$ generated by the simple $\ell$-roots is called the $\ell$-root lattice of $U_q(\tlie g)$ and will be denoted by $\cal Q_q$. Let also $\cal Q_q^+$ be the submonoid generated by the simple $\ell$-roots. Quite clearly $\wt(\gb\alpha_{i,a})=\alpha_i$. Define a partial order on $\cal P_q$ by
$$\gb\mu\le\gb\lambda \qquad{if}\qquad \gb\lambda\gb\mu^{-1}\in\cal Q_q^+.$$

\begin{rem}
The elements $\gb\alpha_{i,a}$ were first defined in \cite{freres:qchar} where they were denoted by $A_{i,aq}$. The term simple $\ell$-root was introduced in \cite{cm:chb} where an alternate definition in terms of an action of the braid group of $\lie g$ on $\cal P_q$ was given. For more details on the $\ell$-weight lattice see \cite[Section 3]{jm:root} and the references therein.
\end{rem}

\section{Finite-dimensional representations}\label{s:fdimreps}

\subsection{Simple Lie algebras}
For the sake of fixing notation, we now review some basic facts about the representation theory of $\lie g$ and $U_q(\lie g)$. For the details see \cite{hum:book} and \cite{cp:book} for instance.

Given a $U_q(\lie g)$-module $V$ and $\mu\in P$, let
$$V_\mu=\{v\in V: k_iv=q^{\mu(h_i)}v \text{ for all } i\in I\}.$$
A nonzero vector $v\in V_\mu$ is called a weight vector of weight $\mu$. If $v$ is a weight vector such that $x_i^+v = 0$ for all $i\in I$, then $v$ is called a highest-weight vector. If $V$ is generated by a highest-weight vector of weight $\lambda$, then $V$ is said to be a highest-weight module of highest weight $\lambda$.
A $U_q(\lie g)$-module $V$ is said to be a weight module if $V=\opl_{\mu\in P}^{} V_\mu$.  Denote by $\cal C_q$ be the category of all finite-dimensional weight modules of $U_q(\lie g)$. Analogous concepts for $\lie g$-modules are defined similarly after setting
$$V_\mu=\{v\in V: hv=\mu(h)v \text{ for all } h\in\lie h\}.$$
Denote by $\cal C$ the category of finite-dimensional $\lie g$-modules.

Let $\mathbb Z[P]$ be the integral group ring over $P$ and denote by $e:P\to\mathbb Z[P], \lambda\mapsto e^\lambda$, the inclusion of $P$ in $\mathbb Z[P]$ so that $e^\lambda e^\mu=e^{\lambda+\mu}$. The character of an object $V$ from $\cal C_q$ or $\cal C$ is defined by
\begin{equation}\label{e:chd}
\ch(V) = \sum_{\mu\in P} \dim(V_\mu) e^\mu.
\end{equation}

The following theorem summarizes the basic facts about the categories $\cal C_q$ and $\cal C$.

\begin{thm}\label{t:ciuqg} Let $V$ be an object either of $\cal C_q$ or of $\cal C$. Then:
\begin{enumerate}
\item $\dim V_\mu = \dim V_{w\mu}$ for all $w\in\cal W$.
\item $V$ is completely reducible.
\item For each $\lambda\in P^+$, the $\lie g$-module $V(\lambda)$ generated by a vector $v$ satisfying
$$x_i^+v=0, \qquad h_iv=\lambda(h_i)v, \qquad (x_i^-)^{\lambda(h_i)+1}v=0,\quad\forall\ i\in I,$$
is irreducible and finite-dimensional. If $V\in\cal C$ is irreducible, then
$V$ is isomorphic to $V(\lambda)$ for some $\lambda\in P^+$.
\item For each $\lambda\in P^+$ the $U_q(\lie g)$-module $V_q(\lambda)$ generated by a vector $v$ satisfying
$$x_i^+v=0, \qquad k_iv=q^{\lambda(h_i)}v, \qquad (x_i^-)^{\lambda(h_i)+1}v=0,\quad\forall\ i\in I,$$
is irreducible and finite-dimensional. If $V\in\cal C_q$ is irreducible, then
$V$ is isomorphic to $V_q(\lambda)$  for some $\lambda\in P^+$.
\item For all $\lambda\in P^+$,  $\ch(V_q(\lambda)) = \ch(V(\lambda))$.
\hfill\qedsymbol
\end{enumerate}
\end{thm}

If $J\subseteq I$ we shall denote by $V_q(\lambda_J)$ the simple $U_q(\lie g_J)$-module of highest weight $\lambda_J$. Similarly $V(\lambda_J)$ denotes the corresponding irreducible $\lie g_J$-module.

\begin{prop}\label{p:restriction}
Let $\lambda\in P^+, J\subseteq I$, and suppose $v\in V_q(\lambda)_\lambda$ (respectively $v\in V(\lambda)_\lambda$) is nonzero. Then $U_q(\lie g_J)v\cong V_q(\lambda_J)$ (respectively $U(\lie g_J)v\cong V(\lambda_J)$).\hfill\qedsymbol
\end{prop}

Assume $\lie g=\lie g_1\oplus\lie g_2$ where $\lie g_j$ are semisimple Lie algebras. Then $P=P_1\times P_2$ where $P_j$ is the weight lattice of $\lie g_j$ for $j=1,2$, and so on. Given $\lambda\in P_j^+$, denote by $V_j(\lambda)$ the irreducible $\lie g_j$-module of highest-weight $\lambda$. If $V_1$ is a $\lie g_1$-module and $V_2$ is a $\lie g_2$-module, then $V_1\otimes V_2$ is naturally a $\lie g$-module.

\begin{prop}\label{p:sumofalgs}
Let $\lambda=(\lambda_1,\lambda_2)\in P^+$ and $\mu=(\mu_1,\mu_2)\in P$. Then:
\begin{enumerate}
\item $V(\lambda)\cong V_1(\lambda_1)\otimes V_2(\lambda_2)$ as $\lie g$-modules.
\item $V(\lambda)_\mu\cong (V_1(\lambda_1)_{\mu_1})\otimes (V_2(\lambda_2)_{\mu_2})$ as $\lie h$-modules.\hfill\qedsymbol
\end{enumerate}
\end{prop}

We will need the following elementary lemma (a proof can be found in \cite[Lemma 2.3]{mou:reslim}).

\begin{lem}\label{l:hwvecs}
Let $V$ be a finite-dimensional $\lie g$-module and suppose $l\in\mathbb Z_{\ge 1}, \nu_k\in P, v_k\in V_{\nu_k}$, for $k=1,\dots,l$, are such that $V=\sum_{k=1}^l U(\lie n^-)v_k$. Fix a decomposition $V= \opl_{j=1}^m V_j$ where $m\in\mathbb Z_{\ge 1},  V_j\cong V(\mu_j)$ for some $\mu_j\in P^+$, and let $\pi_j:V\to V_j$ be the associated projection for $j=1,\dots,m$. Then, there exist distinct $k_1,\dots,k_m\in\{1,\dots,l\}$ such that $\nu_{k_j}=\mu_j$ and $\pi_j(v_{k_j})\ne 0$.\hfill\qedsymbol
\end{lem}

\subsection{Loop algebras}

Let $V$ be a $U_q(\tlie g)$-module. We say that a nonzero vector $v\in V$ is an $\ell$-weight vector if there exists $\gb\lambda\in\cal P_q$ and $k\in\mathbb Z_{>0}$ such that $(\eta-\gb\Psi_\gb\lambda(\eta))^kv=0$ for all $\eta\in U_q(\tlie h)$. In that case, $\gb\lambda$ is said to be the $\ell$-weight of $v$. $V$ is said to be an $\ell$-weight module if every vector of $V$ is a linear combination of $\ell$-weight vectors. In that case, let $V_\gb\lambda$ denote the subspace spanned by all $\ell$-weight vectors of $\ell$-weight $\gb\lambda$.
An $\ell$-weight vector $v$ is said to be a highest-$\ell$-weight vector if $\eta v=\gb\Psi_\gb\lambda(\eta)v$ for every $\eta\in U_q(\tlie h)$ and $x_{i,r}^+v=0$ for all $i\in I$ and all $r\in\mathbb Z$. $V$ is said to be a highest-$\ell$-weight module if it is generated by a highest-$\ell$-weight vector. Denote by $\wcal C_q$ the category of all finite-dimensional $\ell$-weight modules of $U_q(\tlie g)$.
Quite clearly $\wcal C_q$ is an abelian category.

Observe that if $V\in\wcal C_q$, then $V\in\cal C_q$ and
\begin{equation}\label{e:lws}
V_\lambda = \opl_{\gb\lambda:\wt(\gb\lambda)=\lambda}^{} V_\gb\lambda.
\end{equation}
Moreover, if $V$ is a highest-$\ell$-weight module of highest $\ell$-weight $\gb\lambda$, then
\begin{equation}\label{e:hw}
\dim(V_{\wt(\gb\lambda)}) = 1\qquad\text{and}\qquad V_\mu\ne 0 \Rightarrow \mu\le\wt(\gb\lambda).
\end{equation}

Define the concepts of $\ell$-weight vector, etc., for $\tlie g$ in a similar way  and denote by $\wcal C$ the category of all finite-dimensional $\tlie g$-modules. The next proposition is easily established using \eqref{e:hw}.

\begin{prop}
If $V$ is a highest-$\ell$-weight module, then it has a unique proper submodule and, hence, a unique irreducible quotient.\hfill\qedsymbol
\end{prop}

\begin{defn}
Let $\gb\lambda\in \cal P_q^+$ and $\lambda=\wt(\gb\lambda)$. The Weyl module $W_q(\gb\lambda)$ of highest $\ell$-weight $\gb\lambda$ is the $U_q(\tlie g)$-module defined by the quotient of $U_q(\tlie g)$ by the left ideal generated by the elements $x_{i,r}^+, (x_{i,r}^-)^{\lambda(h_i)+1}$, and $\eta-\gb\Psi_\gb\lambda(\eta)$ for every $i\in I, r\in\mathbb Z$, and $\eta\in U_q(\tlie h)$. Denote by $V_q(\gb\lambda)$ the irreducible quotient of $W_q(\gb\lambda)$. The Weyl module $W(\gb\lambda), \gb\lambda\in \cal P^+$, of $\tlie g$ is defined in a similar way. Its irreducible quotient will be denoted by $V(\gb\lambda)$.
\end{defn}

The next theorem was proved in \cite{cp:weyl}.

\begin{thm}\label{t:weyl}
For every $\gb\lambda\in\cal P_q^+$ (resp. $\cal P^+$) the module $W_q(\gb\lambda)$ (resp. $W(\gb\lambda)$) is the universal finite-dimensional $U_q(\tlie g)$-module (resp. $\tlie g$-module) with highest $\ell$-weight $\gb\lambda$. Every simple object of $\wcal C_q$ (resp. $\wcal C$) is highest-$\ell$-weight. \hfill\qedsymbol
\end{thm}

We shall need the following lemma which is a consequence of the proof of Theorem \ref{t:weyl}.

\begin{lem}\label{l:curalggen}
If $V$ is a highest-$\ell$-weight module of $\tlie g$ and $v$ be a highest-$\ell$-weight vector. Then $V=U(\lie g[t])v$.\hfill\qedsymbol
\end{lem}

If $J\subseteq I$ we shall denote by $V_q(\gb\lambda_J)$ the $U_q(\tlie g_J)$-irreducible module of highest $\ell$-weight $\gb\lambda_J$. Similarly $V(\gb\lambda_J)$ denotes  the corresponding irreducible $\tlie g_J$-module. Similar notations for the Weyl modules are defined in the obvious way.

The next theorem was conjectured in \cite{freres:qchar} and proved in \cite{fremuk:qchar}.

\begin{thm}\label{t:cone}
Let $V$ be a quotient of $W_q(\gb\lambda)$ for some $\gb\lambda\in\cal P_q^+$. If $V_\gb\mu\ne 0$, then $\gb\mu\le\gb\lambda$.\hfill\qedsymbol
\end{thm}

Given $V$ in $\wcal C_q$, let $\wtl(V) = \{\gb\mu\in\cal P_q: V_\gb\mu\ne 0\}$. We will need the following proposition proved in \cite[Section 4.8]{mou:reslim}.

\begin{prop}\label{p:lchA}
Suppose $\lie g$ is of type $A$, $\lambda\in P^+$, $\gb\lambda=\prod_{i\in I}\gb\omega_{i,a_i,\lambda(h_i)}$, $\gb\mu\in\wtl(V_q(\gb\lambda))$, and  $\gb\lambda\gb\mu^{-1} = \gb\alpha_{j,b_j}\gb\alpha_{j+1,b_{j+1}}\cdots\gb\alpha_{k,b_k}$ for some $j\le k$ and some $a_i,b_l\in\mathbb C(q)^\times, i\in I, l=j,\dots,k$.
\begin{enumerate}
\item If $\frac{a_{i+1}}{a_i}=q^{\lambda(h_i)+\lambda(h_{i+1})+1}$ for all $i<n$, then $b_k=a_kq^{\lambda(h_k)-1}$.
\item If $\frac{a_{i+1}}{a_i}=q^{-(\lambda(h_i)+\lambda(h_{i+1})+1)}$ for all $i<n$, then $b_j=a_jq^{\lambda(h_j)-1}$.\hfill\qedsymbol
\end{enumerate}
\end{prop}

\subsection{Classical limits}

Denote by $\cal P_\mathbb A^+$ the subset of $\cal P_q$ consisting of $n$-tuples of polynomials with coefficients in $\mathbb A$. Let also $\cal P_\mathbb A^\times $ be the subset of $\cal P_\mathbb A^+$ consisting of $n$-tuples of polynomials whose leading terms are in $\mathbb Cq^{\mathbb Z}\backslash\{0\}=\mathbb A^\times$. Given $\gb\lambda\in\cal P_\mathbb A^+$, let \bgb\lambda\ be the element of $\cal P^+$ obtained from $\gb\lambda$ by evaluating $q$ at $1$.

Recall that an $\mathbb A$-lattice (or form) of a $\mathbb C(q)$-vector space $V$ is a free $\mathbb A$-submodule $L$ of $V$ such that $\mathbb C(q)\otimes_\mathbb A L=V$. If $V$ is a $U_q(\tlie g)$-module, a $U_\mathbb A(\tlie g)$-admissible lattice of $V$ is an $\mathbb A$-lattice of $V$ which is also a $U_\mathbb A(\tlie g)$-submodule of $V$. Given a $U_\mathbb A(\tlie g)$-admissible lattice of a $U_q(\tlie g)$-module $V$, define
\begin{equation}\label{e:clm}
\bar L = \mathbb C\otimes_\mathbb A L,
\end{equation}
where $\mathbb C$ is regarded as an $\mathbb A$-module by letting $q$ act as $1$. Then $\bar L$ is a $\tlie g$-module by Proposition \ref{p:cluq} and $\dim(\bar L)=\dim(V)$. The next theorem is essentially a corollary of the proof of Theorem \ref{t:weyl}.

\begin{thm}\label{t:llattice}
Let $V$ be a nontrivial quotient of $W_q(\gb\lambda)$ for some $\gb\lambda\in \cal P_\mathbb A^\times $, $v$  a highest-$\ell$-weight vector of $V$, and $L=U_\mathbb A(\tlie g)v$. Then, $L$ is a $U_\mathbb A(\tlie g)$-admissible lattice of $V$ and $\ch(\bar L)=\ch(V)$. In particular, $\bar L$ is a quotient of $W(\bgb\lambda)$.\hfill\qedsymbol
\end{thm}

\begin{defn}
Let $\gb\lambda\in \cal P_\mathbb A^\times $, $v$ be a highest-$\ell$-weight vector of $V_q(\gb\lambda)$ and $L=U_\mathbb A(\tlie g)v$. We denote by $\overline{V_q(\gb\lambda)}$ the $\tlie g$-module $\bar L$.
\end{defn}

\section{Minimal affinizations and Beyond}\label{s:min}

\subsection{Classification of minimal affinizations}\label{ss:min} We now review the notion of minimal affinizations of an irreducible $U_q(\lie g)$-module introduced in \cite{cha:minr2}.

Given $\lambda\in P^+$, a $U_q(\tlie g)$-module $V$ is said to be an affinization of $V_q(\lambda)$ if, as a $U_q(\lie g)$-module,
\begin{equation}
V\cong V_q(\lambda)\oplus \opl_{\mu< \lambda}^{} V_q(\mu)^{\oplus m_\mu(V)}
\end{equation}
for some $m_\mu(V)\in\mathbb Z_{\ge 0}$. Two affinizations of $V_q(\lambda)$ are said to be equivalent if they are isomorphic as $U_q(\lie g)$-modules. If $\gb\lambda\in\cal P_q^+$ is such that $\wt(\gb\lambda)=\lambda$, then $V_q(\gb\lambda)$ is quite clearly an affinization of $V_q(\lambda)$. The partial order on $P^+$ induces a natural partial order on the set of  (equivalence classes of) affinizations of $V_q(\lambda)$. Namely, if $V$ and $W$ are affinizations of $V_q(\lambda)$, say that $V\le W$ if one of the following conditions hold:
\begin{enumerate}
\item $m_\mu(V)\le m_\mu(W)$ for all $\mu\in P^+$;
\item for all $\mu\in P^+$ such that $m_\mu(V)>m_\mu(W)$ there exists $\nu>\mu$ such that $m_\nu(V)<m_\nu(W)$.
\end{enumerate}
A minimal element of this partial order is said to be a minimal affinization.

\begin{thm}[{\cite{cp:minsl}}]\label{t:drimin}
Let $\gb\lambda\in\cal P^+_q,\lambda=\wt(\gb\lambda)$, and  $V=V_q(\gb\lambda)$. Suppose $\lie g$ is of type $A$. Then $V$ is a minimal affinization of $V_q(\lambda)$ iff there exist $a\in\mathbb C(q)^\times$ and $\epsilon\in\{1,-1\}$ such that
$$\gb\lambda=\prod_{i=1}^n \gb\omega_{i,a_i,\lambda(h_i)} \qquad\text{with}\qquad a_1 = a \qquad\text{and}\qquad \frac{a_{i+1}}{a_i} = q^{\epsilon(\lambda(h_i)+\lambda(h_{i+1})-1)}$$
for all $i\in I, i<n$.
If $\lie g$ is of type $D$ or $E$, suppose the support of $\lambda$ is contained in a connected subdiagram $J\subseteq I$ of type $A$. Then, $V$ is a minimal affinization of $V_q(\lambda)$ iff $V_q(\gb\lambda_J)$ is a minimal affinization of $V_q(\lambda_J)$. \hfill\qedsymbol
\end{thm}

The next corollaries are easily established (recall from \S\ref{ss:clalg} that $\overline{\rm supp}(\lambda)$ is the minimal connected subdiagram of $I$ containing ${\rm supp}(\lambda)$).

\begin{cor}
Suppose $\lambda\in P^+$ is such that $\overline{\rm supp}(\lambda)$ is of type $A$. Then, $V_q(\lambda)$ has a unique equivalence class of minimal affinizations.\hfill\qedsymbol
\end{cor}

\begin{cor}\label{c:KR}
Given $i\in I$ and $m\in\mathbb Z_{\ge 0}$, the modules $V_q(\gb\omega_{i,a,m}), a\in\mathbb C(q)^\times$, are the only minimal affinizations of $V_q(m\omega_i)$.\hfill\qedsymbol
\end{cor}

The modules $V_q(\gb\omega_{i,a,m})$ are known as Kirillov-Reshetikhin modules.

We now state a few results which were used in the proof of Theorem \ref{t:drimin} and will be useful for us as well. The proofs can be found in \cite{cp:minsl}.

\begin{lem}\label{l:dsa}
Suppose $\emptyset\ne J\subseteq I$ is a connected subdiagram of the Dynkin diagram of $\lie g$. Let $V=V_q(\gb\lambda)$, $v$ a highest-$\ell$-weight vector of $V$, and $V_J = U_q(\tlie g_J)v$. Then, $V_J\cong V_q(\gb\lambda_J)$.\hfill\qedsymbol
\end{lem}

\begin{defn}
Suppose $\lie g$ is of type $A$. A connected subdiagram $J\subseteq I$ is said to be an admissible subdiagram. If $\lie g$ is of type $D$ or $E$, let $i_0\in I$ be the trivalent node. A connected subdiagram $J\subseteq I$ is said to be admissible if $J$ is of type $A$ and $J\backslash\{i_0\}$ is connected.
\end{defn}

\begin{prop}\label{p:admsd}
Suppose $J\subseteq I$ is admissible and that $\gb\lambda\in\cal P_q^+$ is such that $V_q(\gb\lambda)$ is a minimal affinization of $V_q(\lambda)$ where $\lambda=\wt(\gb\lambda)$. Then $V_q(\gb\lambda_J)$ is a minimal affinization of $V_q(\lambda_J)$.\hfill\qedsymbol
\end{prop}

The next proposition was proved in \cite[Proposition 3.7]{mou:reslim}.

\begin{prop}\label{p:fracqz}
Let $\gb\lambda\in P_q^+$ and $\lambda=\wt(\gb\lambda)$. If $V_q(\gb\lambda)$ is a minimal affinization of $V_q(\lambda)$, then there exist $a_i\in\mathbb C(q)^\times,i\in I$, such that $\gb\lambda = \prod_{i\in I}\gb\omega_{i,a_i,\lambda(h_i)}$ and $\frac{a_i}{a_j}\in q^{\mathbb Z}$ for all $i,j\in I$.\hfill\qedsymbol
\end{prop}

\begin{cor}
For every $\lambda\in P^+$ there exist $\gb\lambda\in\cal P_\mathbb A^\times $ such that $V_q(\gb\lambda)$ is a minimal affinization of $V_q(\lambda)$. Moreover, $\bgb\lambda = \gb\omega_{\lambda,a}$ for some $a\in\mathbb C^\times$.\hfill\qedsymbol
\end{cor}

\subsection{Graded characters}\label{ss:limits}

Recall the definition of the maps $\tau_a:\lie g[t]\to\lie g[t]$ from subsection \ref{ss:clalg}.

\begin{defn}
Let $\gb\lambda\in\cal P_\mathbb A^\times , \lambda=\wt(\gb\lambda)$, and $a\in\mathbb C^\times$ be such that $\bgb\lambda = \gb\omega_{\lambda,a}$. The $\lie g[t]$-module $L(\gb\lambda)$ is defined to be the pullback of $\overline{V_q(\gb\lambda)}$ by $\tau_a$.
\end{defn}

It is immediate from Theorem \ref{t:llattice} that
\begin{equation}\label{e:Misoclass}
\ch(L(\gb\lambda))=\ch(V_q(\gb\lambda)).
\end{equation}

Let $V$ be a $\mathbb Z_{\ge 0}$-graded vector space and denote its $r$-th graded piece by $V[r]$. A $\lie g[t]$-module $V$ is said to be $\mathbb Z_{\ge 0}$-graded if $V$ is a $\mathbb Z_{\ge 0}$-graded vector space and $x\otimes t^sv\in V[r+s]$ for every $v\in V[r], x\in\lie g, r,s\in\mathbb Z_{\ge 0}$. Observe that if $V$ is a $\mathbb Z_{\ge 0}$-graded $\lie g[t]$-module, then each graded peace is a $\lie g$-module. Given $s\in\mathbb Z_{\ge 0}$, denote by $V(s)$ the quotient of $V$ by its $\lie g[t]$-submodule $\opl_{r> s}^{} V[r]$. We shall refer to $V(s)$ as the truncation of $V$ at degree $s$. If $V$ is a finite-dimensional $\mathbb Z_{\ge 0}$-graded $\lie g[t]$-module, define the graded character of $V$ by
$$\ch_t(V) =\sum_{r\ge 0} \ch(V[r])\ t^r \ \in\ \mathbb Z[P][t].$$
Let also $m_{\mu,r}(V)$ be the multiplicity of $V(\mu)$ as an irreducible constituent of $V[r]$.

\begin{defn}\label{d:KRres}
Let $m\in \mathbb Z_{\ge 0}$ and $i\in I$. The $\lie g[t]$-module $M(m\omega_i)$ is the quotient of $U(\lie g[t])$ by the left ideal generated by
\begin{equation}\label{e:NKRdef}
\lie n^+[t], \qquad \lie h\otimes t\mathbb C[t], \qquad h_j,\qquad h_i-m,\qquad x_{\alpha_j}^-, \qquad (x_{\alpha_i}^-)^{m+1},\qquad x_{\alpha_i,1}^-\qquad \text{for all } j\ne i.
\end{equation}
Define $T(m\omega_i)$ to be the $\lie g[t]$-submodule of $M(\omega_i)^{\otimes m}$ generated by the top weight space.
\end{defn}

Quite clearly $M(m\omega_i)$ is a $\mathbb Z_{\ge 0}$-graded $\lie g[t]$-module. Given $\lambda\in P^+$ one can consider the modules $A(\lambda)$ defined in \cite{mou:reslim}. These are graded $\lie g[t]$-modules which were proved to be finite-dimensional in \cite[Proposition 3.15]{mou:reslim}. One can proceed similarly to prove that the modules $M(m\omega_i)$ are finite-dimensional. Moreover, it was proved in  \cite[Proposition 5.2.5]{per} that $A(m\omega_i)\cong M(m\omega_i)$ (for a general simple Lie algebra $\lie g$). We shall not need the modules $A(\lambda)$ in this paper.

Given $i\in I, m,r\in \mathbb Z_{\ge 0}$, let $v_{i,m}$ be the image of $1$ in $M(m\omega_i)$ and set
\begin{equation}
R(i,m,r) = \{\alpha\in R^+: x_{\alpha,r}^-v_{i,m}=0\}.
\end{equation}
Since $(\lie h\otimes t\mathbb C[t])v_{i,m}=0$, it follows that
\begin{equation}\label{e:Rgrows}
R(i,m,r)\subseteq R(i,m,s) \qquad\text{for all }s\ge r.
\end{equation}
In particular, it follows that $M(0)$ is the trivial representation and $R^+(i,0,s)=R^+$ for all $i\in I, s\in\mathbb Z_{\ge 0}$. Now, given $\lambda\in P^+$ and $r\in\mathbb Z_{\ge 0}$, set
\begin{equation}
R(\lambda, r) = \bigcap_{i\in I} R(i,\lambda(h_i),r).
\end{equation}
Notice $R(m\omega_i,r)=R(i,m,r)$ for all $i\in I$ and $m,r\in\mathbb Z_{\ge 0}$.

\begin{defn}\label{d:N}
Let $\lambda\in P^+$. The $\lie g[t]$-module $M(\lambda)$ is the quotient of $U(\lie g[t])$ by the left ideal generated by
\begin{equation}\label{e:Ndef}
\lie n^+[t], \qquad \lie h\otimes t\mathbb C[t], \qquad h_i-\lambda(h_i), \qquad (x_{\alpha_i}^-)^{\lambda(h_i)+1}, \qquad x_{\alpha,r}^-
\end{equation}
for all $i\in I, r\in\mathbb Z_{\ge 0}$, and $\alpha\in R(\lambda,r)$.  Define $T(\lambda)$ to be the $\lie g[t]$-submodule of $\otm_{i\in I}^{} M(\lambda(h_i)\omega_i)$ generated by the top weight space.
\end{defn}

Definitions \ref{d:KRres} and \ref{d:N} of $M(m\omega_i)$ coincide since $R(m\omega_i,r)=R(i,m,r)$ for all $i\in I, m,r\in\mathbb Z_{\ge 0}$. The modules $M(\lambda)$ are clearly $\mathbb Z_{\ge 0}$-graded. It follows from \cite[Proposition 3.13]{mou:reslim} that $M(\lambda)$ is a quotient of the module $A(\lambda)$ of \cite{mou:reslim} and, hence, finite-dimensional. Moreover, one easily sees that $T(\lambda)$ is a graded quotient of $M(\lambda)$ for all $\lambda\in P^+$ (the details can be found in \cite[Proposition 5.2.10]{per}).

\begin{prop}[{\cite[Proposition 3.21]{mou:reslim}}]\label{p:TqL}
Let $\gb\lambda\in\cal P_\mathbb A^\times $ be such that $V_q(\gb\lambda)$ is a minimal affinization of $V_q(\lambda)$ where $\lambda=\wt(\gb\lambda)$. Then, $T(\lambda)$ is a quotient of $L(\gb\lambda)$.\hfill\qedsymbol
\end{prop}

The following is the main conjecture of \cite{mou:reslim}.

\begin{con}\label{c:mou}
Let $\lambda\in P^+$. Then, $M(\lambda)\cong T(\lambda)$. Moreover, if $\overline{\rm supp}(\lambda)$ is of type $A$ and $\gb\lambda\in\cal P_\mathbb A^\times $ is such that $V_q(\gb\lambda)$ is a minimal affinization of $V_q(\lambda)$, then, $M(\lambda)\cong L(\gb\lambda)$.
\end{con}

For the rest of the subsection assume that $\lie g$ is of type $E_6$ and that the nodes of the Dynkin diagram are labeled as in the introduction. We now state our main results.

\begin{thm}\label{t:main}
Let $\lambda\in P^+$ be such that $\lambda(h_3)=0$. Suppose that either $\{2,4\}\nsubseteq\supp(\lambda)$ or $\supp(\lambda)\subseteq\{2,4,6\}$. Then:
\begin{enumerate}
\item The first isomorphism in Conjecture \ref{c:mou} holds.
\item The second isomorphism in Conjecture \ref{c:mou} holds provided that $\overline\supp(\lambda)$ is of type $A$.
\end{enumerate}
\end{thm}
Notice that part (a) of Theorem \ref{t:main} and Proposition \ref{p:TqL} together with the following proposition which will be proved in Subsection \ref{ss:LqM} imply part (b) of Theorem \ref{t:main}.

\begin{prop}\label{p:LqM}
Let $\lambda\in P^+$ be such that is of type $A$. Then, $L(\gb\lambda)$ is a quotient of $M(\lambda)$.
\end{prop}

As a byproduct of the proof of Theorem \ref{t:main} we are able to compute $\ch_t(M(\lambda))$ for $\lambda$ as in the theorem. In particular, we compute $\ch(V_q(\gb\lambda))$ for all $\gb\lambda\in\cal P_q^+$ such that $\wt(\gb\lambda)$ satisfies the hypothesis of part (b) of the theorem. Let us now present these formulas and, along the way, explain the strategy of the proof of Theorem \ref{t:main}(a).

Fix $\lambda\in P^+$ and, given $\mu\in P$ and $r\in\mathbb Z_{\ge 0}$, set
$$m_{\mu,r}= m_{\mu,r}(M(\lambda)) \qquad\text{and}\qquad t_{\mu,r}=m_{\mu,r}(T(\lambda)).$$
We have already seen that $t_{\mu,r}\le m_{\mu,r}$. Therefore, in order to prove the first isomorphism of Conjecture \ref{c:mou}, it suffices to show that
\begin{equation}\label{e:main}
m_{\mu,r}\le t_{\mu,r} \qquad\text{for all}\qquad \mu\in P^+, r\in\mathbb Z_{\ge 0}.
\end{equation}
For $\gbr r\in\mathbb Z^6$, set
$$\wt(\gbr r)=\lambda-r_1(\omega_2-\omega_5)-r_2(\omega_4-\omega_1)-r_3(\omega_2-\omega_4+\omega_5)- r_4(\omega_1-\omega_2+\omega_4) -r_5(\omega_2-\omega_3+\omega_4)-r_6\omega_6$$
and
$$\gr(\gbr r)= r_1+r_2+r_3+r_4+r_5+r_6.$$
Let also
$$\cal A=\{\gbr r\in\mathbb Z_{\ge 0}^6: r_6\le m_6,r_3\le m_5, r_4\le m_1, r_1+r_3+r_5\le m_2, r_2+r_4+r_5\le m_4\},$$
$$\cal A_{\mu} = \{\gbr r\in\cal A:\wt(\gbr r)=\mu\}, \qquad \cal A_r = \{\gbr r\in\cal A:\gr(\gbr r)=r\}, \quad\text{and}\quad \cal A_{\mu,r}=\cal A_\mu\cap\cal A_r.$$
The omission of the dependence of $\wt$ and $\cal A$ on $\lambda$ in the notation will not create confusion.
One easily checks that the function $\wt:\mathbb Z^6\to P$ is injective and, if $\gbr r\in\cal A$, then $\wt(\gbr r)\in P^+$. In particular,
\begin{equation}\label{e:multfree}
|\cal A_\mu|\le 1 \qquad\text{for all}\qquad \mu\in P^+.
\end{equation}
The basic idea for proving \eqref{e:main} is the same one used in \cite{cm:kr,cm:krg,mou:reslim}. Namely, in Subsection \ref{ss:ub}, we will use the defining relations of $M(\lambda)$ to show that,
\begin{equation}\label{e:ub}
\text{if}\qquad \lambda(h_3)=0, \qquad\text{then}\qquad m_{\mu,r}\le |\cal A_{\mu,r}|.
\end{equation}
Moreover,  for $\lambda$ as in Theorem \ref{t:main}, by performing some explicit computations in $T(\lambda)$, we show in  Subsection \ref{ss:lb} that
\begin{equation}\label{e:lb}
t_{\mu,r}\ge |\cal A_{\mu,r}|.
\end{equation}
Clearly \eqref{e:ub} and \eqref{e:lb} together imply \eqref{e:main}. Moreover,
\begin{equation}\label{e:chM}
\ch_t(M(\lambda)) = \sum_{\gbr r\in\cal A} \ch(V(\wt(\gbr r)))t^{\gr(\gbr r)}
\end{equation}
for all $\lambda$ as in Theorem \ref{t:main}. In particular, for $\lambda$ as in Theorem \ref{t:main}(b) and $\gb\lambda\in\cal P_q^+$ such that $V_q(\gb\lambda)$ is a minimal affinization of $V_q(\lambda)$, we have
\begin{equation}\label{e:chL}
\ch(V_q(\gb\lambda)) = \sum_{\gbr r\in\cal A} \ch(V(\wt(\gbr r))).
\end{equation}

\begin{rem}
Similar results in the case that $\lie g$ is of classical type or $G_2$ were obtained in \cite{cm:kr,cm:krg,mou:reslim} (however, the definition of the modules $T(m\omega_i)$ requires some extra care in the non simply laced case). Equation \eqref{e:chM} (and similar ones for general $\lie g$) was predicted in \cite{jap:rem} in the case that $\lambda=m\omega_i$ for some $i\in I, m\in\mathbb Z_{\ge 0}$. However, the meaning of the gradation in \cite{jap:rem} is related to the quantum context, whereas here it appears by computing the classical limit. It is not clear to us why these two gradations coincide. The formulas in \cite{jap:rem} were obtained by assuming the Kirillov-Reshetikhin conjecture whose proof was later completed in \cite{her:krc}. Our results give an alternate proof of these formulas for $\lie g$ of type $E_6$ and $i\ne 3$. As mentioned in the introduction, $M(m\omega_3)$ is not multiplicity free in general. Using the methods of this paper, we are able to prove that the isotypical components of $M(m\omega_3)[r]$ are exactly as given by \cite{jap:rem}. However, so far we could only obtain an upper bound for $m_{\mu,r}$ which is most often larger than the actual value of $m_{\mu,r}$.
\end{rem}

We end this subsection by reviewing a construction used in \cite[\S2.6]{cm:kr} which will be useful for us as well. Let $V_r$,  $0\le r\le k$,  be  $\lie g$-modules  such that
\begin{equation}\label{e:condun}
{\rm Hom}_{\lie g}(\lie g\otimes V_r,  V_{r+1})\ne 0,\ \ {\rm Hom}_{\lie g}(\wedge^2(\lie g)\otimes V_r,  V_{r+2})=0,\ \ 0\le r\le k-1,
\end{equation}
where we assume that $V_{k+1}=0$. Fix non-zero elements $p_r\in{\rm Hom}_{\lie g}(\lie g\otimes V_r, V_{r+1})$, $0\le r\le k-1$, and set $p_k=0$. It is easily checked that the following formulas extend the canonical  $\lie g$-module structure to a graded  $\lie  g[t]$-module structure on $V=\oplus_{r=1}^k V_r$:
\begin{equation}\label{e:construction}
(x\otimes t)w=p_r(x\otimes w), \quad  (x\otimes t^s)w=0,\quad\text{for all}\quad  x\in\lie g, w\in V_r,  1\le r\le k, s\ge 2.
\end{equation}
Moreover, $V[r]\cong V_{r}$ for all $0\le r\le k$. Also, if $V_0=U(\lie g)w_0$ and the maps $p_r$ for $r<k$ are all surjective, then $V=U(\lie n^-[t])w_0$.

\subsection{Projectivity}\label{ss:proj}
If $\overline\supp(\lambda)$ is not of type $A$, then Proposition \ref{p:LqM} is probably false. In fact, most likely, $M(\lambda)$ is then a proper quotient of $L(\gb\lambda)$. We now explain the motivation for studying the modules $M(\lambda)$ beyond the cases associated to minimal affinizations from the perspective of \cite{cg:proj}. We begin with following straightforward lemma which has been implicitly used in \cite{cg:proj}.

\begin{lem}\label{l:truncate}
Let $r\in\mathbb Z_{> 0}$ and $V$ be a $\lie g[t]$-module generated by a vector $v$ satisfying $(\lie g\otimes t^r\mathbb C[t])v=0.$ Then,  $(\lie g\otimes t^r\mathbb C[t])V=0.$
\end{lem}

\begin{proof}
Let $x\in\lie g, s\ge r$, and $w=(x_1\otimes t^{r_1})\cdots (x_2\otimes t^{r_m})v$ for some $m,r_j\in\mathbb Z_\ge 0, x_j\in \lie g, j=1,\dots,m$.  We proceed by induction on $m$. If $m=0$, we have $(x\otimes t^s)w=0$ by hypothesis. Assume $m>0$, let $w'=(x_2\otimes t^{r_2})\cdots (x_m\otimes t^{r_m})v$ and assume, by induction hypothesis, that $(y\otimes t^s)w'=0$ for all $y\in\lie g, s\ge r$. Then, given $x\in\lie g$ and $s\ge r$, we have
\begin{align*}
(x\otimes t^s)w=(x_1\otimes t^{r_1})(x\otimes t^s)w'+([x,x_1]\otimes t^{s+s_1})w'.
\end{align*}
Both summands are zero by the induction hypothesis on $m$.
\end{proof}

The next proposition follows immediately from the above lemma and the definition of $M(\lambda)$.

\begin{prop}
Let $\lambda\in P^+$ and $r>0$ be such that $R(\lambda,r)=R^+$. Then, $(\lie g\otimes t^r\mathbb C[t])M(\lambda)=0.$\hfill\qedsymbol
\end{prop}

If $V$ is a $\lie g[t]$-module as in Lemma \ref{l:truncate}, then the canonical projection $\lie g[t]\to \lie g[t:r]:=\lie g[t]/\lie g\otimes t^r\mathbb C[t]$ induces a $\lie g[t:r]$-module structure on $V$. Chari and Greenstein in \cite{cg:kos,cg:proj} initiated the study of the category $\cal G_2$ of graded $\lie g[t:2]$-modules with finite-dimensional graded pieces (they do not assume $\lie g$ is simply laced). Given a subset $\Gamma$ of $P^+\times\mathbb Z_\ge 0$, they consider the full subcategories $\cal G_2(\Gamma)$ of $\cal G_2$ consisting of modules $V$ such that $V(\mu)$ is an irreducible constituent of $V[r]$ only if $(\mu,r)\in\Gamma$.
In particular, they consider subsets $\Gamma$ of the following form. Given $\Psi\subseteq R^+$ and $\lambda\in P$, set
$$\Gamma(\lambda,\Psi) = \{(\mu,r)\in P\times\mathbb Z_\ge 0: \lambda-\mu=\sum_{\beta\in \Psi} n_\beta\beta, n_\beta\in\mathbb Z_\ge 0, \sum_{\beta\in\Psi} n_\beta = r\}.$$
Notice that $(\lambda,0)\in\Gamma(\lambda,\Psi)$ for any choice of $\Psi$ and that $\Gamma(\lambda,\emptyset)=\{(\lambda,0)\}$. If we regard $V(\lambda)$ as a module for $\lie g[t:2]$ by pulling back the canonical projection $\lie g[t:2]\to \lie g[t:1]=\lie g$, then $V(\lambda)$ is an object of $\cal G_2(\Gamma(\lambda,\Psi))$. The full strength of the results of \cite{cg:proj} is realized when $\Psi$ is either empty or of the form $\Psi_\nu$ for some $\nu\in P$ where
$$\Psi_\nu = \{\alpha\in R^+: (\alpha,\nu) = \max\{(\beta,\nu):\beta\in R^+\}\}$$
and $(\cdot,\cdot)$ is the bilinear form on $P\times P$ induced from the Killing form of $\lie g$.

For $\lambda\in P^+$ such that $R(\lambda,2)=R^+$, set $\Psi^\lambda=R^+\backslash R(\lambda,1)$. The following theorem is a particular case of \cite[Theorem 1]{cg:proj}.
\begin{thm}\label{t:proj}
Let $\lambda\in P^+$ be such that $R(\lambda,2)=R^+$ and suppose that either $\Psi^\lambda=\emptyset$ or $\Psi^\lambda=\Psi_\nu$ for some $\nu\in P$. Then, $M(\lambda)$ is the projective cover of $V(\lambda)$ in the category $\cal G_2(\Gamma(\lambda,\Psi^\lambda))$.\hfill\qedsymbol
\end{thm}

For $\lambda$ as in Theorem \ref{t:proj}, \cite[Theorem 2]{cg:proj} gives a formula for computing the graded character of $M(\lambda)$ by induction on the cardinality of the set $\Gamma(\lambda,\Psi^\lambda)$.

Let us return to the case that $\lie g$ is of type $E_6$.
It follows from the proof of Theorem \ref{t:main} (see Lemma \ref{l:coord}  below)  that $M(\lambda)$ is a module as in Lemma \ref{l:truncate} with $r=3$. Moreover, if $\lambda(h_3)=0$, then we can take $r=2$.

\begin{lem}\label{l:proj}
Let $\lambda\in P^+$ be such that $\lambda(h_3)=0$ and $\{2,4\}\nsubseteq\supp(\lambda)$. Then, either $\Psi^\lambda=\emptyset$ or there exists $\nu\in P$ such that $\Psi^\lambda=\Psi_\nu$.
\end{lem}

\begin{proof}
Recalling that $(\alpha_i,\nu) = \frac{1}{2}(\alpha_i,\alpha_i)\nu(h_i)$ and using the characterizations of $R(\lambda,1)$ given by \eqref{e:coord}, one easily checks by inspection of Table 1 below that
\begin{enumerate}
\item $\supp(\lambda)\subseteq\{1,5\}\Rightarrow \Psi^\lambda=\emptyset$.
\item $6\in \supp(\lambda)\subseteq\{1,5,6\}\Rightarrow \Psi^\lambda=\Psi_{\omega_6}$.
\item $2\in \supp(\lambda)\subseteq\{1,2,5,6\}\Rightarrow \Psi^\lambda=\Psi_{\omega_2}$.
\item $4\in \supp(\lambda)\subseteq\{1,4,5,6\}\Rightarrow \Psi^\lambda=\Psi_{\omega_4}$.
\end{enumerate}
Clearly $\lambda$ satisfies the hypothesis of the lemma iff it satisfies one of the conditions (a)-(d) above.
\end{proof}

This immediately implies the following corollary of Theorem \ref{t:proj}.

\begin{cor}\label{c:Mproj}
Let $\lambda$ be as in Lemma \ref{l:proj}.  Then, $M(\lambda)$ is the projective cover of $V(\lambda)$ in the category $\cal G_2(\Gamma(\lambda,\Psi^\lambda))$.\hfill\qedsymbol
\end{cor}

Similarly to the proof of Lemma \ref{l:proj}, one easily checks that if $\{2,4\}\subseteq \supp(\lambda)$, then $\Psi^\lambda\ne\emptyset$ and $\Psi^\lambda\ne\Psi_{\nu}$ for all $\nu\in P$. Therefore, $\lambda$ satisfies the hypothesis of Theorem \ref{t:proj} iff it satisfies the hypothesis of Lemma \ref{l:proj}. It follows that every $\lambda$ as in Theorem \ref{t:proj} satisfies the hypothesis of Theorem \ref{t:main}. On the other hand,
if $\lambda$ satisfies the hypothesis of Theorem \ref{t:main} but not the one of Theorem \ref{t:proj}, then  $\{2,4\}\subseteq\supp(\lambda)\subseteq\{2,4,6\}$. In this case, we cannot conclude that $M(\lambda)$ is a projective object in some subcategory of $\cal G_2$ nor can we use \cite[Theorem 2]{cg:proj} to compute its graded character.

\begin{rem}
It is worth remarking that we will perform most of the proof of \eqref{e:lb} using only the hypothesis $\lambda(h_3)=0$. This provides some evidence that Conjecture \ref{c:mou} holds in complete generality. In particular, we conjecture that \eqref{e:chM} is the graded character of $M(\lambda)$ for all $\lambda\in P^+$ such that $\lambda(h_3)=0$.
\end{rem}

\section{Proofs}\label{s:proof}

\subsection{On characters for type $A_2$}
We now record some lemmas about the characters of certain finite-dimensional $\lie{sl}_3$-modules which will be needed in the proof of \eqref{e:lb}. To simplify some formulas, we introduce the notation of divided powers. If $A$ is an associative algebra, $x\in A$, and $r\in\mathbb Z_{\ge 0}$, set $x^{(r)} = \frac{1}{r!}x^r$.

We will make use of the following result on representations of the 3-dimensional Heisenberg algebra which will also be used in the proof of \eqref{e:ub}. Thus, consider the three-dimensional Heisenberg Lie algebra $\lie H$ spanned by elements $x,y,z$ where $z$ is central and $[x,y]=z$. Part (a) of the following lemma is standard while a proof of part (b) can be found in {\cite[Lemma 1.5]{cm:krg}}.

\begin{lem}\label{l:heis}
Let $r,s\in\mathbb Z_{\ge 0}$, $V$ a representation of $\lie H$, and suppose $0\ne v\in V$ is such that $x^rv=0$.
\begin{enumerate}
\item The following identity holds in $U(\lie H)$: $x^{(r)}y^{(s)} = \sum_{k=0}^{\min\{r,s\}} z^{(k)}y^{(s-k)}x^{(r-k)}$.
\item For all $k\in\mathbb Z_{\ge 0}$,  the element $y^sz^kv$ is in the span of elements of the form $x^ay^bz^cv$ with  $0\le c<r, a+c=k$, and $b+c=k+s$. Moreover, if $xv=0$, then $y^szv = \frac{1}{s+1}\ xy^{s+1}v$.\hfill\qedsymbol
\end{enumerate}
\end{lem}

Recall that $U(\lie n^-)$ is $Q^+$-graded and denoted by $U(\lie n^-)_\eta$ the piece of degree $\eta$. For the remainder of this subsection we assume $\lie g=\lie{sl}_3$ and $I=\{1,2\}$. Observe that the map $\lie n^-\to \lie H$ given by $x^-_i\mapsto x$ and $x^-_j\mapsto y$, where $i,j\in I$ are distinct, is an isomorphism.
\begin{lem}\label{l:basisUheis}
Let $i,j\in I,i\ne j$, and $\eta = k_i\alpha_i+k_j\alpha_j\in Q^+$. Then $\{(x_i^-)^{(r)}(x_j^-)^{(k_j)}(x_i^-)^{(k_i-r)}: 0\le r\le\min\{k_i,k_j\}\}$ is a basis of $U(\lie n^-)_\eta$.
\end{lem}

\begin{proof}
Since $\dim(U(\lie n^-)_\eta)=p(\eta)=\min\{k_i,k_j\}+1$, it suffices to show that this set is linearly independent. Let us write $x=x_i^-, y=x_j^-$, and $z=[x,y]$. Then, by  part (a) of Lemma \ref{l:heis} we have
$$x^{(r)}y^{(k_j)}x^{(k_i-r)} = \sum_{k=0}^{\min\{r,k_j\}} \binom{k_i-k}{r-k} z^{(k)}y^{(k_j-k)}x^{(k_i-k)}.$$
One now easily uses the PBW theorem to prove that these vectors, with $0\le r\le\min\{k_i,k_j\}$, are linearly independent.
\end{proof}

\begin{lem}\label{l:sl3}
Let $\lambda=m_1\omega_1+m_2\omega_2\in P^+, 0\le k_1\le m_1, 0\le k_2\le m_2$, and $\mu=\lambda-k_1\alpha_1-k_2\alpha_2$. Then, $\dim(V(\lambda)_\mu)=\min\{k_1,k_2\}+1$.
\end{lem}

\begin{proof} Straightforward using Kostant's multiplicity formula (cf. \cite[Proposition 5.3.10]{per}).
\end{proof}

\begin{lem}\label{l:basis}
Let $V$ be a finite-dimensional $\lie g$-module, $l\in\mathbb Z_{\ge 1}$, and $\mu_1,\dots,\mu_l\in P^+$. Assume $\mu_l<\mu_s$ for all $s<l$, write $\eta_s = \mu_s-\mu_l = k_{s,1}\alpha_1+k_{s,2}\alpha_2$, and suppose $k_{s,i}\le \mu_s(h_i), i\in I$. Suppose also that there exists $v_s\in V_{\mu_s}$ such that $V=\sum_{s=1}^l U(\lie n^-)v_s$. Let $i,j\in I$ be distinct. Then, $V\cong \opl_{s=1}^l V(\mu_s)$ iff the vectors $(x_i^-)^{(r)}(x_j^-)^{(k_{s,j})}(x_i^-)^{(k_{s,i}-r)}v_s$ for $s=1,\dots,l$ and $0\le r\le \min\{k_{s,1},k_{s,2}\}$ are linearly independent.
\end{lem}

\begin{proof}
By Lemma \ref{l:sl3} we have $\dim(V(\mu_s))_{\mu_l}= \min\{k_{s,1},k_{s,2}\}+1$ and by Lemma \ref{l:hwvecs} there exists $m\le l$ and $s_1,\dots, s_m$ such that $V\cong \opl_{r=1}^m V(\mu_{s_r})$. Hence,
$$\dim(V_{\mu_l}) = \sum_{r=1}^m \dim(V(\mu_{s_r})_{\mu_l})=\sum_{r=1}^m (\min\{k_{s_r,1},k_{s_r,2}\}+1).$$
The if part follows since the cardinality of the set $\{(s,r):s=1,\dots,l, 0\le r\le \min\{k_{s,1},k_{s,2}\}\}$ is $\sum_{s=1}^l (\min\{k_{s,1},k_{s,2}\}+1)$.

Conversely, assume that $V\cong \opl_{s=1}^l V(\mu_s)$ and let $V_s, s=1,\dots,l$, be a submodule of $V$ isomorphic to $V(\mu_s)$ and such that $V=\opl_{s=1}^l V_s$. Let also $\pi_s:V\to V_s$ be the associated projection. By Lemma \ref{l:hwvecs} we can assume $\pi_s(v_s)$ is a highest weight vector of $V_s$. Observe that the set $(x_i^-)^{(r)}(x_j^-)^{(k_{s,j})}(x_i^-)^{(k_{s,i}-r)}\pi_s(v_s)$ with $0\le r\le \min\{k_{s,1},k_{s,2}\}$ is a basis of $(V_s)_{\mu_l}$. Indeed, the set $(x_i^-)^{(r)}(x_j^-)^{(k_{s,j})}(x_i^-)^{(k_{s,i}-r)}$ is a basis for $U(\lie n^-)_{\eta_s}$ by Lemma \ref{l:basisUheis}. In particular, the vectors $(x_i^-)^{(r)}(x_j^-)^{(k_{s,j})}(x_i^-)^{(k_{s,i}-r)}\pi_s(v_s)$ with $0\le r\le \min\{k_{s,1},k_{s,2}\}$ span $(V_s)_{\mu_l}$. Since we already know that $\dim((V_s)_{\mu_l}) = \min\{k_{s,1},k_{s,2}\}+1$, the claim follows. Let $a_{r,s}\in\mathbb C$ be such that
$$\sum_{s=1}^l\sum_{r=0}^{\min\{k_{s,1},k_{s,2}\}} a_{r,s}(x_i^-)^{(r)}(x_j^-)^{(k_{s,j})}(x_i^-)^{(k_{s,i}-r)}v_s=0.$$
Given $1\le t\le l$, we get
\begin{align*}
\pi_t(\sum_{s=1}^l & \sum_{r=0}^{\min\{k_{s,1},k_{s,2}\}} a_{r,s}(x_i^-)^{(r)}(x_j^-)^{(k_{s,j})}(x_i^-)^{(k_{s,i}-r)}v_s) =\\
&\sum_{r=0}^{\min\{k_{t,1},k_{t,2}\}} a_{r,t}(x_i^-)^{(r)}(x_j^-)^{(k_{t,j})}(x_i^-)^{(k_{t,i}-r)}\pi_t(v_t)=0.
\end{align*}
It follows that $a_{r,t}=0$ for all $t=1,\dots,l$ and $0\le r\le \min\{k_{t,1},k_{t,2}\}$.
\end{proof}

\begin{lem}\label{l:sl3heis}
Let $a,b,c,m\in\mathbb Z_{\ge 0}, i,j\in I, j\ne i, \lambda=m\omega_i$, and $v\in V(\lambda)_\lambda\backslash\{0\}$. Then,
$$(x_i^-)^a(x_j^-)^b(x_i^-)^cv\ne 0 \quad\Leftrightarrow\qquad b\le c \qquad\text{and}\qquad a+c\le m.$$
Moreover,
$$(x_i^-)^a(x_j^-)^b(x_i^-)^cv=\left(\prod_{s=1}^a\frac{c+s-b}{c+s}\right)\ (x_j^-)^b(x_i^-)^{a+c}v.$$
\end{lem}

\begin{proof}
From the $\lie{sl}_2$ representation theory we have $(x_i^-)^cv\ne 0$ iff $c\le m$.
Since $x_j^+(x_i^-)^cv= 0$ and $h_j(x_i^-)^cv=c(x_i^-)^cv$, it follows from the $\lie{sl}_2$ representation theory once more that $(x_j^-)^b(x_i^-)^cv\ne 0$ iff $b\le c$ (and $c\le m$). Notice that this together with the second statement implies the first statement. We prove the second statement by induction on $a\ge 0$. The case $a=0$ is obvious. The induction step will however depend on the knowledge of the case $a=1$. For convenience set $x=x_j^-, y=x_i^-$, and $z = [x,y]$. Using the well-known commutation relation in $U(\lie n^-)$
\begin{equation*}
yx^b = x^by -bx^{b-1}z
\end{equation*}
we get
\begin{equation*}
yx^by^cv = x^by^{c+1}v - bx^{b-1}y^czv = x^by^{c+1}v - \frac{b}{c+1}\ x^by^{c+1}v = \frac{c+1-b}{c+1}\ x^by^{c+1}v
\end{equation*}
where, in the second equality, we used that $xv=0$ and the last statement of Lemma \ref{l:heis}. The case $a=1$ follows. Then, for $a>1$, using the induction hypothesis we get
\begin{equation*}
y^ax^by^cv = y(y^{a-1}x^by^cv) = \left(\prod_{s=1}^{a-1}\frac{c+s-b}{c+s}\right)yx^by^{c+a-1}v.
\end{equation*}
Since, by the case $a=1$, we have
\begin{equation*}
yx^by^{c+a-1}v = \left(\frac{c+a-b}{c+a}\right) x^by^{a+c}v
\end{equation*}
the second statement follows.
\end{proof}

\begin{rem}
Notice that if $b\le c$ the number $\prod_{s=1}^a\frac{c+s-b}{c+s}$ is a positive rational number.
\end{rem}

\subsection{Root data}

Henceforth we assume $\lie g$ is of type $E_6$, set $\lambda=\sum_{i\in I} m_i\omega_i\in P^+$, and assume $\gb\lambda\in\cal P_\mathbb A^\times$ is such that $V_q(\gb\lambda)$ is a minimal affinization of $V_q(\lambda)$. We will need the expression of every positive root in terms of the simple roots and of some of them in terms of the fundamental weights. These expressions are given by Tables 1 and 2 below, respectively.

\vspace{5pt}
\centerline{\bf Table 1}
\begin{multicols}{2}
\begin{itemize}
\item[] $\beta_1=\alpha_1+\alpha_2$ \vspace{0pt}
\item[] $\beta_2=\alpha_4+\alpha_5$  \vspace{0pt}
\item[] $\beta_3=\alpha_2+\alpha_3$ \vspace{0pt}
\item[] $\beta_4=\alpha_3+\alpha_4$ \vspace{0pt}
\item[] $\beta_5=\alpha_3+\alpha_6$  \vspace{0pt}
\item[] $\beta_6=\alpha_1+\alpha_2+\alpha_3$ \vspace{0pt}
\item[] $\beta_7=\alpha_3+\alpha_4+\alpha_5$  \vspace{0pt}
\item[] $\beta_8=\alpha_2+\alpha_3+\alpha_6$  \vspace{0pt}
\item[] $\beta_9=\alpha_3+\alpha_4+\alpha_6$  \vspace{0pt}
\item[] $\beta_{10}=\alpha_2+\alpha_3+\alpha_4$\vspace{0pt}
\item[] $\beta_{11}=\alpha_1+\alpha_2+\alpha_3+\alpha_6$ \vspace{0pt}
\item[] $\beta_{12}=\alpha_3+\alpha_4+\alpha_5+\alpha_6$ \vspace{0pt}
\item[] $\beta_{13}=\alpha_2+\alpha_3+\alpha_4+\alpha_6$ \vspace{0pt}
\item[] $\beta_{14}=\alpha_1+\alpha_2+\alpha_3+\alpha_4$ \vspace{0pt}
\item[] $\beta_{15}=\alpha_2+\alpha_3+\alpha_4+\alpha_5$ \vspace{0pt}
\item[] $\beta_{16}=\alpha_1+\alpha_2+\alpha_3+\alpha_4+\alpha_5$ \vspace{0pt}
\item[] $\beta_{17}=\alpha_1+\alpha_2+\alpha_3+\alpha_4+\alpha_6$ \vspace{0pt}
\item[] $\beta_{18}=\alpha_2+\alpha_3+\alpha_4+\alpha_5+\alpha_6$ \vspace{0pt}
\item[] $\beta_{19}=\alpha_1+\alpha_2+\alpha_3+\alpha_4+\alpha_5+\alpha_6$  \vspace{0pt}
\item[] $\beta_{20}=\alpha_2+2\alpha_3+\alpha_4+\alpha_6$ \vspace{0pt}
\item[] $\beta_{21}=\alpha_1+\alpha_2+2\alpha_3+\alpha_4+\alpha_6$ \vspace{0pt}
\item[] $\beta_{22}=\alpha_2+2\alpha_3+\alpha_4+\alpha_5+\alpha_6$ \vspace{0pt}
\item[] $\beta_{23}=\alpha_1+\alpha_2+2\alpha_3+\alpha_4+\alpha_5+\alpha_6$  \vspace{0pt}
\item[] $\beta_{24}=\alpha_1+2\alpha_2+2\alpha_3+\alpha_4+\alpha_6$ \vspace{0pt}
\item[] $\beta_{25}=\alpha_2+2\alpha_3+2\alpha_4+\alpha_5+\alpha_6$ \vspace{0pt}
\item[] $\beta_{26}=\alpha_1+2\alpha_2+2\alpha_3+\alpha_4+\alpha_5+\alpha_6$ \vspace{0pt}
\item[] $\beta_{27}=\alpha_1+\alpha_2+2\alpha_3+2\alpha_4+\alpha_5+\alpha_6$ \vspace{0pt}
\item[] $\beta_{28}=\alpha_1+2\alpha_2+2\alpha_3+2\alpha_4+\alpha_5+\alpha_6$ \vspace{0pt}
\item[] $\beta_{29}=\alpha_1+2\alpha_2+3\alpha_3+2\alpha_4+\alpha_5+\alpha_6$ \vspace{0pt}
\item[] $\beta_{30}=\alpha_1+2\alpha_2+3\alpha_3+2\alpha_4+\alpha_5+2\alpha_6$
\end{itemize}
\end{multicols}

\vspace{5pt}

\centerline{\bf Table 2}
\begin{multicols}{2}
\begin{itemize}
\item[] $\alpha_1=2\omega_1-\omega_2$ \vspace{0pt}
\item[] $\alpha_2=2\omega_2-\omega_1-\omega_3$ \vspace{0pt}
\item[] $\alpha_3=2\omega_3-\omega_2-\omega_4-\omega_6$ \vspace{0pt}
\item[] $\alpha_4=2\omega_4-\omega_3-\omega_5$ \vspace{0pt}
\item[] $\alpha_5=2\omega_5-\omega_4$ \vspace{0pt}
\item[] $\alpha_6=2\omega_6-\omega_3$ \vspace{0pt}
\item[] $\beta_{23}=\omega_1-\omega_2+\omega_3-\omega_4+\omega_5$ \vspace{0pt}
\item[] $\beta_{24}=\omega_2-\omega_5$ \vspace{0pt}
\item[] $\beta_{25}=\omega_4-\omega_1$ \vspace{0pt}
\item[] $\beta_{26}=\omega_2-\omega_4+\omega_5$ \vspace{0pt}
\item[] $\beta_{27} = \omega_1-\omega_2+\omega_4$
\item[] $\beta_{28}=\omega_2-\omega_3+\omega_4$ \vspace{0pt}
\item[] $\beta_{29}=\omega_3-\omega_6$ \vspace{0pt}
\item[] $\beta_{30}=\omega_6$
\end{itemize}
\end{multicols}

\subsection{A smaller set of relations for $M(\lambda)$}\label{ss:lessrel}
In order to prove Proposition \ref{p:LqM}, we need a version of \cite[Proposition 4.6]{mou:reslim}.

\begin{prop}\label{p:lessrel}
Suppose that either $m_3\ne 0$ or $\overline{\supp}(\lambda)$ is of type $A$. Then, $M(\lambda)$ is isomorphic to the $\lie g[t]$-module $N(\lambda)$ generated by a vector $v$ satisfying
\begin{equation*}
h_iv\ =\ m_iv \qquad\text{and}\qquad \lie n^+[t]v\ =\ \lie h\otimes t\mathbb C[t]v\ =\ (x_{\alpha_i}^-)^{m_i+1}v\ =\ x_{\alpha,1}^-v=0
\end{equation*}
for all $\alpha\in R^+_1 := \{\alpha\in R^+: \alpha=\sum_{i\in I} n_i\alpha_i \text{ with } n_i\le 1 \text{ for all } i\in I\}=R^+\backslash\{\beta_j:j\ge 20\}$.
\end{prop}

\begin{proof}
It follows from Lemma \ref{l:coord} that $R^+_1\subseteq R(\lambda,1)$ and, hence, $M(\lambda)$ is a quotient of $N(\lambda)$. Let us now show that, under the hypothesis assumed on $\lambda$, we have an epimorphism in the opposite direction. Thus, we need to show that $x_{\alpha,r}^-v=0$ for all $\alpha\in R(\lambda,r)$. In fact, after  \eqref{e:Rgrows}, given $\alpha\in R^+$, it suffices to show that
\begin{equation}\label{e:lessrel}
x_{\alpha,r_{\alpha}}^-v=0 \qquad\text{where}\qquad r_\alpha=\min\{r:\alpha\in R(\lambda,r)\}.
\end{equation}
If $r_\alpha=0$ this follows immediately  from the defining relations of $N(\lambda)$ since they clearly imply that $U(\lie g)v\cong V(\lambda)$. If $\alpha\in R^+_1$ equation \eqref{e:lessrel} is again immediate from the defining relations of $N(\lambda)$. Therefore, we need to prove \eqref{e:lessrel} for $\alpha\in R^+\backslash R^+_1$ only. Notice also that Lemma \ref{l:coord} implies that $r_\alpha\le 3$ for all $\alpha\in R^+$.

Assume first that $m_3\ne 0$. It then follows from \eqref{e:coord} that $R^+_1=R(\lambda,1)$ and \eqref{e:lessrel} is immediate for all $\alpha$ such that $r_\alpha=1$. Equation \eqref{e:coord} also implies that $R(\lambda,2) = \{\beta_j:20\le j\le 28\}$ and $R(\lambda,3)=\{\beta_{29},\beta_{30}\}$. Therefore, we are left to show that $x_{\beta_j,2}^-v=0$ for all $20\le j\le 28$ and $x_{\beta_j,3}^-v=0$ for all $29\le j\le 30$. This follows from the following commutation relations together with \eqref{e:lessrel} for $\alpha$ such that $r_\alpha\le 1$:
\begin{alignat*}{5}
&x^-_{\beta_{20},2}= [x^-_{\alpha_3,1},x^-_{\beta_{13},1}], \qquad\qquad&& x^-_{\beta_{21},2}=[x^-_{\alpha_3,1},x^-_{\beta_{17},1}], \qquad\qquad&& x^-_{\beta_{22},2}= [x^-_{\alpha_3,1},x^-_{\beta_{18},1}],\\
&x^-_{\beta_{23},2}=[x^-_{\alpha_3,1},x^-_{\beta_{19},1}], \qquad\qquad&&  x^-_{\beta_{24},2}=[x^-_{\beta_3,1},x^-_{\beta_{17},1}], \qquad\qquad&& x^-_{\beta_{25},2}=[x^-_{\beta_4,1},x^-_{\beta_{18},1}],\\
&x^-_{\beta_{26},2}=[x^-_{\beta_3,1},x^-_{\beta_{19},1}], \qquad\qquad&&  x^-_{\beta_{27},2}=[x^-_{\beta_4,1},x^-_{\beta_{19},1}], \qquad\qquad&& x^-_{\beta_{28},2}=[x^-_{\beta_{10},1},x^-_{\beta_{19},1}],\\
&x^-_{\beta_{29},3}=[x^-_{\alpha_3,1},x^-_{\beta_{28},2}], \qquad\qquad&& x^-_{\beta_{30},3}=[x^-_{\beta_{18},1},x^-_{\beta_{21},2}].
\end{alignat*}

Now, assume $m_3=0$. In this case, $r_\alpha\le 2$ for all $\alpha\in R^+$. We consider separately the cases $\supp(\lambda)\subseteq\{1,2,4,5\}$ and $\supp(\lambda)\subseteq\{1,2,6\}$ (the case $\supp(\lambda)\subseteq\{4,5,6\}$ follows from the latter by the symmetry of the Dynkin diagram). Thus, assume $\supp(\lambda)\subseteq\{1,2,6\}$ and consider the following relations
\begin{alignat*}{5}
&x^-_{\beta_{20},1}= [x^-_3,x^-_{\beta_{13},1}], \qquad\qquad&& x^-_{\beta_{21},1}=[x^-_3,x^-_{\beta_{17},1}], \qquad\qquad&& x^-_{\beta_{22},1}= [x^-_3,x^-_{\beta_{18},1}],\\
&x^-_{\beta_{23},1}=[x^-_3,x^-_{\beta_{19},1}], \qquad\qquad&& x^-_{\beta_{25},1}=[x^-_{\beta_4},x^-_{\beta_{18},1}],\qquad\qquad&&  x^-_{\beta_{27},1}=[x^-_{\beta_4},x^-_{\beta_{19},1}].
\end{alignat*}
Since $\alpha_3,\beta_4\in R(\lambda,0)$ in this case, it follows that $x_{\beta_j,1}^-v=0$ for all $20\le j\le 27, j\ne 24,26$. If $m_2=0$, we need to show that $x_{\beta_j,1}^-v=0$ for $j\in\{24,26,28,29\}$ and $x_{\beta_{30},r}^-v=0$ where $r=1$ if $m_6=0$ and $r=2$ otherwise. Since, in this case, $\beta_3,\beta_{10}\in R(\lambda,0)$, the former follows from the following relations
\begin{equation*}
x^-_{\beta_{24},1}=[x^-_{\beta_3},x^-_{\beta_{17},1}], \qquad x^-_{\beta_{26},1}=[x^-_{\beta_3},x^-_{\beta_{19},1}], \qquad x^-_{\beta_{28},1}=[x^-_{\beta_{10}},x^-_{\beta_{19},1}],\qquad x^-_{\beta_{29},1}=[x^-_3,x^-_{\beta_{28},1}].
\end{equation*}
The latter follows from the relations
\begin{equation*}
x^-_{\beta_{30},1}=[x^-_{\beta_{18}},x^-_{\beta_{21},1}] \qquad\text{and}\qquad x^-_{\beta_{30},2}=[x^-_{\beta_{18},1},x^-_{\beta_{28},1}]
\end{equation*}
using that $\beta_{18}\in R(\lambda,0)$ if $m_6=0$.

Assume $\supp(\lambda)\subseteq\{1,2,4,5\}$. As in the previous case, one sees that $x_{\beta_j,1}^-v=0$ for all $20\le j\le 23$. If both $m_2$ and $m_4$ are nonzero, we are left to show that $x_{\beta_j,2}^-v=0$ for all $24\le j\le 30$. For $24\le j\le 28$, this is done as in the case $m_3\ne 0$ while for $j=29,30$ this then follows from the relations
\begin{equation*}
 x^-_{\beta_{29},2}=[x^-_3,x^-_{\beta_{28},2}] \qquad\text{and}\qquad x^-_{\beta_{30},2}=[x^-_{\beta_{18},1},x^-_{\beta_{21},1}].
\end{equation*}
If $m_2=0$ and $m_4\ne 0$, we need to show that $x_{\beta_{24},1}^-v=x_{\beta_{26},1}^-v=0$. This is done as in the case $\supp(\lambda)\subseteq\{1,2,6\}$. The case $m_2\ne0$ and $m_4=0$ is treated similarly. In particular, if $m_4=0$ we have $x_{\beta_{25},1}^-v=x_{\beta_{27},1}^-v=0$. Finally, if $m_2=m_4= 0$, we need to prove in addition that $x_{\beta_j,1}^-v=0$ for $j=28,29,30$. This is done as in the case $\supp(\lambda)\subseteq\{1,6\}$.
\end{proof}

\subsection{Quantized relations}\label{ss:LqM} The goal of this subsection is to prove Proposition \ref{p:LqM}.
We proceed as in the proof of \cite[Proposition 3.22]{mou:reslim} where a similar statement for orthogonal Lie algebras was proved. First we record several previously proved results which will be used in the proof.

\begin{lem}[{\cite[Lemma 4.18]{mou:reslim}}]\label{l:x1}
Suppose $w$ is a highest-$\ell$-weight vector of $V_q(\gb\omega_{i,a,m})$ for some $i\in I,a\in\mathbb C(q)^\times$, and $m\in\mathbb Z_{\ge 0}$. Then, $x_{i,1}^-w = aq^mx_i^-w.$\hfill\qedsymbol
\end{lem}

The following proposition follows from the results of \cite[Section 6]{cha:braid}.

\begin{prop}\label{p:cyclic}
Let $l\in\mathbb Z_{\ge 1}, i_j\in I, m_j\in\mathbb Z_{\ge 1}, a_j\in\mathbb C(q)^\times$ for $j=1,\dots, l$. If $\frac{a_{j}}{a_{k}}\notin q^{\mathbb Z_{> 0}}$ for $j>k$, then $V_q(\gb\omega_{i_1,a_1,m_1})\otimes \cdots\otimes V_q(\gb\omega_{i_l,a_l,m_l})$ is a highest-$\ell$-weight module.\hfill\qedsymbol
\end{prop}

\begin{cor}[{\cite[Corollary 4.4]{mou:reslim}}]\label{c:tpkrmin}
Let $\lambda\in P^+, a_i\in\mathbb C(q)^\times, i\in I$, and $\gb\lambda=\prod_{i\in I} \gb\omega_{i,a_i,\lambda(h_i)}$. Then, there exists an ordering $i_1,\dots, i_n$ of $I$ such that $V_q(\gb\lambda)$ is isomorphic to the $U_q(\tlie g)$-submodule of $V_q(\gb\omega_{i_1,a_{i_1},\lambda(h_{i_1})})\otimes\cdots\otimes V_q(\gb\omega_{i_n,a_{i_n},\lambda(h_{i_n})})$ generated by the top weight space.\hfill\qedsymbol
\end{cor}

\begin{prop}[{\cite[Proposition 3.13]{mou:reslim}}]\label{p:admevrel}
Suppose $\gb\lambda\in\cal P_\mathbb A^\times $ is such that $V_q(\gb\lambda)$ is a minimal affinization and that $J\subseteq I$ is an admissible subdiagram. Let $v$ be a highest-$\ell$-weight vector of $V=\overline{V_q(\gb\lambda)}, \lambda=\wt(\gb\lambda)$, and $a\in\mathbb C^\times$ be such that $\bgb\lambda = \gb\omega_{\lambda,a}$. Then $x_{\alpha,r}^-v=a^rx_\alpha^-v$ for every $\alpha\in R^+_J$.\hfill\qedsymbol
\end{prop}

If $\alpha\in R^+_J$ for some admissible diagram $J$, we shall refer to $\alpha$ as an admissible root.

\begin{proof}[Proof of Proposition \ref{p:LqM}]
Let $a\in\mathbb C$ be such that $\bar{\gb\lambda}=\gb\omega_{\lambda,a}$. We fix a highest-$\ell$-weight vector $v$ of $V=V_q(\gb\lambda)$ and $a_i\in\mathbb A^\times, i\in I$, such that $\gb\lambda=\prod_{i\in I} \gb\omega_{i,a_i,m_i}$. Let also $\bar v$ be the image of $v$ in $\overline V$ and $v'$ be the image of $\bar v$ in $L(\gb\lambda)$. By Proposition \ref{p:lessrel}, we need to show that $x_{\alpha,1}^-v'=0$ for all $\alpha\in R^+_1$. This is equivalent to showing that
\begin{equation}\label{e:LqM}
x_{\alpha,1}^-\bar v= a\bar v \qquad\text{for all}\qquad \alpha\in R^+_1.
\end{equation}
By Proposition \ref{p:admevrel}, \eqref{e:LqM} holds if $\alpha$ is an admissible root. Therefore, it remains to show that
\begin{equation}\label{e:LqMbeta}
x_{\beta_j,1}^-\bar v= a\bar v \qquad\text{for all}\qquad 7< j<20.
\end{equation}

Assume first that $\supp(\lambda)\subseteq \{1,2,3,4,5\}$. In this case $\alpha_6\in R(\lambda,0)$ and \eqref{e:LqMbeta} with $j\in\{8,9,11,12\}$ follows from the following relations
\begin{equation*}
x_{\beta_8,1}^-=[x_6^-,x_{\beta_3,1}^-], \quad x_{\beta_9,1}^-=[x_6^-,x_{\beta_4,1}^-], \quad x_{\beta_{11},1}^-=[x_6^-,x_{\beta_6,1}^-], \quad x_{\beta_{12},1}^-=[x_6^-,x_{\beta_7,1}^-]
\end{equation*}
together with the fact that $\beta_3,\beta_4,\beta_6$, and $\beta_7$ are admissible roots. Next, assume that we have proved \eqref{e:LqMbeta} for $j\in\{10,14,15,16\}$. Then, \eqref{e:LqMbeta} for the remaining values of $j$ follows from the following relations
\begin{equation*}
x_{\beta_{13},1}^-=[x_6^-,x_{\beta_{10},1}^-], \quad x_{\beta_{17},1}^-=[x_6^-,x_{\beta_{14},1}^-], \quad x_{\beta_{18},1}^-=[x_6^-,x_{\beta_{15},1}^-], \quad x_{\beta_{19},1}^-=[x_6^-,x_{\beta_{16},1}^-].
\end{equation*}
In order to prove \eqref{e:LqMbeta} for $j\in\{10,14,15,16\}$, it suffices to find elements $X_j, X_{j,1}\in U_\mathbb A(\tlie n^-)$ such that
\begin{equation}\label{e:LqMq}
\overline{X_j}=x_{\beta_j}^-, \quad \overline{X_{j,1}}=x_{\beta_j,1}^-, \quad\text{and}\quad X_{j,1}v=a_j(q)X_jv + x_jv
\end{equation}
for some $a_j(q)\in\mathbb A$ and $x_j\in U_\mathbb A(\lie g)$ satisfying $a_j(1)=a$ and $\overline x_j=0$. We prove the existence of such elements assuming
\begin{equation}\label{e:minpos}
a_{i+1} = a_iq^{m_i+m_{i+1}+1} \qquad\text{for all}\qquad i<5.
\end{equation}
The case $a_{i+1} = a_iq^{-(m_i+m_{i+1}+1)}, i<5,$ is proved similarly using part (b) of Proposition \ref{p:lchA} instead of part (a). Let $i_0=\max\{i\in I:m_i\ne 0\}$ (in the case $a_{i+1} = a_iq^{-(m_i+m_{i+1}+1)}, i<5,$ we would use $i_0=\min\{i\in I:m_i\ne 0\}$). The relations $X_{j,1}v=a_j(q)X_jv + x_jv$ of \eqref{e:LqMq} are the quantized relations alluded to in the title of this subsection.

Let $\gb\lambda'$ be such that $\gb\lambda=\gb\lambda'\gb\omega_{i_0,a_{i_0},m_{i_0}}$. Let also $v_1,v_2$ be highest-$\ell$-weight vectors of $V_q(\gb\lambda')$ and $V_q(\gb\omega_{i_0,a_{i_0},m_{i_0}})$, respectively.
By \eqref{e:minpos}, Proposition \ref{p:cyclic}, and Corollary \ref{c:tpkrmin}, the assignment $v\mapsto v_1\otimes v_2$ extends to an isomorphism $V\cong U_q(\tlie g)(v_1\otimes v_2)\subseteq V_q(\gb\lambda')\otimes V_q(\gb\omega_{i_0,a_{i_0},m_{i_0}})$. Henceforth, we identify $v$ with $v_1\otimes v_2$.
We write down the proof of the existence of elements as in \eqref{e:LqMq} for $j=16$ assuming $i_0=5$ (the other cases are proved similarly and the computations are simpler). Set
\begin{equation*}
X_{14} = [x_4^-,[x_3^-,[x_1^-,x_2^-]]], \qquad X_{16} = [x_5^-,X_{14}], \quad\text{and}\quad X_{16,1} = [x_{5,1}^-,X_{14}].
\end{equation*}
Quite clearly, $X_{16}, X_{16,1}\in U_\mathbb A(\tlie g)$ satisfy the first two identities in \eqref{e:LqMq}. By Lemmas \ref{l:comcom} and \ref{l:comult}, modulo an element of the form $xv$ with $x\in U_\mathbb A(\tlie g)\otimes U_\mathbb A(\tlie g)$ such that $\bar x=0$, we have
\begin{align*}
X_{16}v & = x_5^-X_{14}(v_1\otimes v_2)- X_{14}x_5^-(v_1\otimes v_2)\\
& = x_5^-((X_{14}v_1)\otimes v_2)-X_{14}(v_1\otimes (x_5^-v_2))\\
& = (x_5^-X_{14}v_1)\otimes (k_5^{-1}  v_2) + (X_{14}v_1)\otimes (x_5^- v_2)\\
& - (X_{14}v_1)\otimes ((k_1k_2k_3k_4)^{-1}x_5^-v_2)- v_1\otimes(X_{14}x_5^-v_2)\\
& = q^{-m_5}(x_5^-X_{14}v_1)\otimes v_2 + (1-q^{-m_5})(X_{14}v_1)\otimes(x_5^-v_2) - v_1\otimes(X_{14}x_5^-v_2)
\end{align*}
while
\begin{align*}
X_{16,1}v & = x_{5,1}^-X_{14}(v_1\otimes v_2)- X_{14}x_{5,1}^-(v_1\otimes v_2)\\
& = x_{5,1}^-((X_{14}v_1)\otimes v_2)-X_{14}(v_1\otimes (x_{5,1}^-v_2))\\
& = (x_{5,1}^-X_{14}v_1)\otimes (k_5  v_2) + (X_{14}v_1)\otimes (x_{5,1}^- v_2)\\
& - (X_{14}v_1)\otimes ((k_1k_2k_3k_4)^{-1}x_{5,1}^-v_2)- v_1\otimes(X_{14}x_{5,1}^-v_2)\\
& = q^{m_5}(x_{5,1}^-X_{14}v_1)\otimes v_2 + (1-q^{-m_5})(X_{14}v_1)\otimes(x_{5,1}^-v_2) - v_1\otimes(X_{14}x_{5,1}^-v_2).
\end{align*}
Using Lemma \ref{l:x1} we get
\begin{align*}
X_{16,1}v &  = q^{m_5}(x_{5,1}^-X_{14}v_1)\otimes v_2 + (1-q^{-m_5})(X_{14}v_1)\otimes(a_5q^{m_5}x_5^-v_2) - v_1\otimes(X_{14}(a_5q^{m_5}v_2))\\
&= a_5q^{m_5}X_{16}v + q^{m_5}(x_{5,1}^-X_{14}v_1)\otimes v_2 - a_5(x_5^-X_{14}v_1)\otimes v_2.
\end{align*}
Since $a_{16}(q):=a_5q^{m_5}$ satisfies $a_{16}(1)=a$, in order to prove that $X_{16}$ and $X_{16,1}$ satisfy the last identity of \eqref{e:LqMq}, it suffices to show that
\begin{equation}\label{e:LqMend}
q^{m_5}(x_{5,1}^-X_{14}v_1)\otimes v_2 = a_5(x_5^-X_{14}v_1)\otimes v_2.
\end{equation}
Notice that $x_{5,r}^+X_{14}v_1=0$ for all $r\in\mathbb Z$ and let $W$ be the $U_q(\tlie g_5)$-submodule of $V_q(\gb\lambda')$ generated by $X_{14}v_1$. Then, by Proposition \ref{p:lchA}(a), $W$ is a highest-$\ell$-weight module with highest $\ell$-weight $\gb\omega_{5,a_4q^{m_4}}$. It then follows from Lemma \ref{l:x1} that
$$x_{5,1}^-X_{14}v_1 = a_4q^{m_4+1}x_5^-X_{14}v_1.$$
This and \eqref{e:minpos} imply \eqref{e:LqMend}.

The case $\supp(\lambda)\subseteq \{1,2,3,6\}$ is dealt with similarly and the case $\supp(\lambda)\subseteq \{3,4,5,6\}$ then follows using the symmetry of the Dynkin diagram. We omit the details.
\end{proof}

\subsection{Upper bounds}\label{ss:ub} In this subsection we prove \eqref{e:ub}. Let $v\in M(\lambda)_\lambda$ be nonzero.
\begin{lem}\label{l:coord}
For every $i\in I, m\in \mathbb Z_{\ge 0}$, and $\alpha=\sum_{j\in I} a_j\alpha_j\in R^+$ we have $\alpha\in R(i,m,a_i)$. In particular:
\begin{enumerate}
\item $R(1,m,1)=R(5,m,1)=R^+$.
\item $R(6,m,1)\supseteq R^+\setminus\{\beta_{30}\}$ and $R(6,m,2)=R^+$.
\item $R(2,m,1)\supseteq R^+\setminus\{\beta_{24},\beta_{26},\beta_{28},\beta_{29},\beta_{30}\}$ and $R(2,m,2)=R^+$.
\item $R(4,m,1)\supseteq R^+\setminus\{\beta_{25},\beta_{27},\beta_{28},\beta_{29},\beta_{30}\}$ and $R(4,m,2)=R^+$.
\item $R(3,m,1)\supseteq R^+\setminus\{\beta_j:j\ge 20\}, R(3,m,2)\supseteq R^+\setminus\{\beta_{29},\beta_{30}\}$, and $R(3,m,3)=R^+$.
\end{enumerate}
\end{lem}

\begin{proof}
Statements (a)-(e) follow from the first statement by inspection of Table 1. Conversely, clearly items (a)-(e) together imply the first statement. The proof is analogous to that of \cite[Proposition 1.2]{cha:fer}  (see also \cite[Lemma 5.2.8]{per}). We omit the details.
\end{proof}

Observe that the above Lemma together with \eqref{e:Rgrows} imply
\begin{equation*}
x_{\alpha_i,r}^-v=x_{\beta_j,r}^-v=x_{\beta_k,s}^-v=0
\end{equation*}
for all $i\in I, j<20, k<29, r\ge 1, s\ge 2$ and $R(\lambda,3)=R^+.$
Let $R'(i,m,r)$ be the set on the right-hand-side of the inclusion symbol of the appropriate item of Lemma \ref{l:coord}. It will follow from Section \ref{ss:kr} below that, if $m>0$, then
\begin{equation}\label{e:coord}
R(i,m,r)=R'(i,m,r).
\end{equation}
Set $R'(\lambda,r) = \cap_{i\in I} R'(i,m_i,r)$ and let $\lie r(\lambda)$ be the subspace of $\lie g[t]$ spanned by $\{x_{\alpha,1}^-, x_{\beta,2}^-: \alpha\in R^+\backslash R'(\lambda,1), \beta\in R^+\backslash R'(\lambda,2)\}$ which is clearly an abelian ideal of $\lie n^-[t]$. Since we are assuming $m_3=0$, we have $R'(\lambda,2)=R^+$ and, therefore, $\lie r(\lambda)$ is the subspace of $\lie g[t]$ spanned by $\{x_{\alpha,1}^-: \alpha\in R^+\backslash R'(\lambda,1)\}$. Since $R(\lambda,r)=R^+$ for all $r\ge 2$ by \eqref{e:Rgrows}, a straightforward application of the PBW Theorem implies
\begin{equation}\label{e:Ugen}
M(\lambda) = U(\lie n^-[t])v = U(\lie n^-)U(\lie r(\lambda))v.
\end{equation}
Moreover,
\begin{equation}\label{e:R1245}
R(\lambda,1)\supseteq R^+\setminus\{\beta_{24}, \beta_{25}, \beta_{26},
\beta_{27}, \beta_{28},\beta_{29},\beta_{30}\}.
\end{equation}
by Lemma \ref{l:coord} and, therefore,
\begin{equation}\label{e:Ufull1245}
M(\lambda)=U(\lie n^-) U(x^-_{\beta_{30},1}) U(x^-_{\beta_{29},1}) U(x^-_{\beta_{28},1})  U(x^-_{\beta_{27},1}) U(x^-_{\beta_{26},1}) U(x^-_{\beta_{25},1}) U(x^-_{\beta_{24},1})v.
\end{equation}
We now apply Lemma \ref{l:heis} to prove that
\begin{equation}\label{e:U12456}
M(\lambda)=U(\lie n^-) U(x^-_{\beta_{30},1})U(x^-_{\beta_{28},1}) U(x^-_{\beta_{27},1}) U(x^-_{\beta_{26},1}) U(x^-_{\beta_{25},1}) U(x^-_{\beta_{24},1})v.
\end{equation}
Indeed, let $x=x_3^-$, $y=x^-_{\beta_{28},1}$, $z=x^-_{\beta_{29},1}$ which generates a three-dimensional Heisenberg subalgebra of $\lie g[t]$. Since $xv=0$, it follows from Lemma \ref{l:heis} that $(x^-_{\beta_{28},1})^{r}(x^-_{\beta_{29},1})^{s} v$ is a multiple of $(x^-_3)^{s} (x^-_{\beta_{28},1})^{r+s} v$ for every $r,s\in\mathbb Z_\ge 0$.  Since $[x_3^-,x^-_{\beta_j,1}]=0$ for all $24\le j\le 30, j\ne 28$, \eqref{e:U12456} follows.

Given $\gbr r=(r_1,r_2,r_3,r_4,r_5,r_6)\in\mathbb Z_{\ge 0}^6$, set
$$\gbr x_{_\gbr r} = (x^-_{\beta_{30},1})^{r_6}(x^-_{\beta_{28},1})^{r_5}(x^-_{\beta_{27},1})^{r_4}(x^-_{\beta_{26},1})^{r_3}(x^-_{\beta_{25},1})^{r_2} (x^-_{\beta_{24},1})^{r_1}$$
so that \eqref{e:U12456} is equivalent to
\begin{equation}\label{e:12456}
M(\lambda)=\sum_{\gbr r\in\mathbb Z_{\ge 0}^6} U(\lie n^-)\gbr x_{_\gbr r}v.
\end{equation}
Recall the definition of $\wt(\gbr r)$ in Subsection \ref{ss:limits} and use Table 2 to observe that $\gbr x_{_\gbr r}v\in M(\lambda)[\gr(\gbr r)]_{\wt(\gbr r)}$.

Consider the Heisenberg subalgebra of $\lie g[t]$ generated by $\{x_1^-, x^-_{\beta_{25},1},$ $x^-_{\beta_{27},1}\}$. Since $(x_1^-)^{m_1+1}v=0$ and $[x_1^-,x^-_{\beta_j,1}]=0$ for all $24\le j\le 30, j\ne 25$, it follows from Lemma \ref{l:heis} that we can restrict the sum of \eqref{e:12456} to  $\gbr r\in\mathbb Z_{\ge 0}^6$ such that $r_4\le m_1$. Similarly, by working with the Heisenberg subalgebra generated by $\{x_5^-, x^-_{\beta_{24},1},x^-_{\beta_{26},1}\}$ we can assume $r_3\le m_5$.

Next, we show that we can restrict the sum of \eqref{e:12456} to  $\gbr r\in\mathbb Z_{\ge 0}^6$ such that $r_1+r_3+r_5\le m_2$ and $r_2+r_4+r_5\le m_4$. By contradiction, assume this is not the case. It then follows from Lemma \ref{l:hwvecs} that there exists $\gbr r\in\mathbb Z_{\ge 0}^6$ satisfying either $r_1+r_3+r_5> m_2$ or $r_2+r_4+r_5> m_4$ and such that $V(\wt(\gbr r))$ is an irreducible summand of $M(\lambda)$. Moreover, the injectivity of $\wt:\mathbb Z^6\to P$ implies that the projection of $\gbr x_{_\gbr r}v$ on this summand is non zero. Fix such $\gbr r$ and suppose $r_1+r_3+r_5> m_2$ (the other case follows from the symmetry of the Dynkin diagram). Let $\gbr s=(r_1,r_2,r_3,0,r_4+r_5,r_6)$ and notice that
\begin{equation}\label{e:ub12456}
(x_2^+)^{r_4}\gbr x_{_\gbr s}v = c\gbr x_{_\gbr r}v \qquad\text{for some}\qquad c\in\mathbb C^\times.
\end{equation}
This easily follows from the relations
\begin{gather*}
[x_2^+,x_{\beta_j,1}^-]=0 \qquad\text{for all}\qquad 24\le j\le 30, \quad j\ne 26,28, \\
[x_2^+,x_{\beta_{26},1}^-]=x_{\beta_{23},1}^-,  \qquad [x_2^+,x_{\beta_{28},1}^-]=x_{\beta_{27},1}^-,\\
[x_{\beta_{23},1}^-,x_{\beta_j,1}^-]=0 \qquad\text{for all}\qquad 24\le j\le 30,  \qquad\text{and}\qquad x_{\beta_{23},1}^-v=0.\\
\end{gather*}
It follows from \eqref{e:ub12456} that the projection of $\gbr x_{_\gbr s}v$ on $V(\wt(\gbr r))$ is non zero and, hence, $V(\wt(\gbr r))_{\wt(\gbr s)}\ne 0$. We claim that this is a contradiction.
Indeed, notice that $\wt(\gbr s)(h_2) = (m_2-r_1-r_3-r_4-r_5)$. Hence, $\sigma_2\wt(\gbr s) = \wt(\gbr s)-(m_2-r_1-r_3-r_4-r_5)\alpha_2$ is a weight of $V(\wt(\gbr r))$. Here, $\sigma_2\in\cal W$ is the simple reflection associated with $\alpha_2$. Since $\wt(\gbr s)=\wt(\gbr r)-r_4\alpha_2$, it follows that
\begin{equation*}
\sigma_2\wt(\gbr s) = \wt(\gbr r) + (r_1+r_3+r_5-m_2)\alpha_2> \wt(\gbr r),
\end{equation*}
contradicting $V(\wt(\gbr r))_{\sigma_2\wt(\gbr s)}\ne 0$.

So far we proved that the sum in \eqref{e:12456} can be restricted to $\gbr r\in\mathbb Z_{\ge 0}^6$ such that $r_4\le m_1, r_3\le m_5, r_1+r_3+r_5\le m_2$, and $r_2+r_4+r_5\le m_4$. Now, observe that, for such $\gbr r$, $\wt(\gbr r)\in P^+$ iff $\gbr r\in\cal A$.  Therefore, by Lemma \ref{l:hwvecs}, we must have a surjective homomorphism of $\lie g$-modules
\begin{equation*}
\opl_{\gbr r\in\cal A_r}^{} V(\wt(\gbr r))\to M(\lambda)[r]
\end{equation*}
for every $r\in\mathbb Z_{\ge 0}$ and \eqref{e:ub} follows.

\begin{rem}
Let $w\in T(\lambda)_\lambda$ be nonzero and notice that, since $T(\lambda)$ is a quotient of $M(\lambda)$, equations \eqref{e:12456} remain valid after replacing $M(\lambda)$ by $T(\lambda)$ on the left-hand-side and $v$ by $w$ on the right-hand-side.
\end{rem}

\subsection{The Kirillov-Reshetikhin case}\label{ss:kr}

In this subsection we assume $\lambda=m_i\omega_i$ for some $i\in I, i\ne 3,$ and prove \eqref{e:lb} in this case. As mentioned earlier, for such $\lambda$, \eqref{e:chM} (and hence \eqref{e:lb}) follows from \cite{her:krc,jap:rem} (see also \cite{cha:fer}). However, in order to prove \eqref{e:lb} for more general $\lambda$ later, we will need further details about this case than just \eqref{e:lb}. Hence, we consider it separately. We split the proof in cases according to the value of $i$. We keep denoting by $v$ a nonzero vector in $M(\lambda)_\lambda$.

\subsubsection{} Assume $i=1$ or $i=5$ and notice that Lemma \ref{l:coord} implies $\lie r(\lambda)=0$ in this case. Hence, $M(\lambda)=U(\lie n^-)v$ and it follows that $M(\lambda)$ is isomorphic to the pullback of $V(\lambda)$ by the map $\lie g[t]\to \lie g, x\otimes f(t)\mapsto f(0)x$. Since $\cal A=\{\lambda\}$ in this case, \eqref{e:lb} follows.

\subsubsection{} Now suppose $i=6$. Notice that $\cal A_r=\{(0,0,0,0,0,r)\}$ for all $0\le r\le m_6$ and $\cal A_r=\emptyset$ otherwise. Since $\wt((0,0,0,0,0,r))=(m_6-r)\omega_6$, \eqref{e:lb} becomes
\begin{equation}\label{e:lb6}
t_{(m_6-r)\omega_6,r}\ne 0 \quad\text{for all}\quad 0\le r\le m_6.
\end{equation}
We begin proving this in the case $m_6=1$ in which case we have $T(\lambda)=M(\lambda)$ by definition. Observe that ${\rm Hom}_\lie g(\lie g\otimes V(\omega_6), V(\omega_6))\ne 0$ which is true since $V(\omega_6)$ is isomorphic to the adjoint representation. Hence, we can apply the construction given by \eqref{e:construction} with $V_0=V(\omega_6)$ and $V_1=V(0)$. One easily checks that the highest weight vector $w_0$ of $V_0$ satisfies the relations satisfied by $v$ and, hence, the module $V$ constructed in this way is a quotient of $M(\omega_6)$. Since $V[0]\cong V(\omega_6)$ and $V[1]\cong V(0)$, \eqref{e:lb6} follows. Moreover, we clearly have $x_{\beta_{30},1}^-w_0\ne 0$ (otherwise the map $p_0$ would be zero) and, hence, $x_{\beta_{30},1}^-v\ne 0$. In particular, \eqref{e:coord} holds for $i=6$.

For $m_6>1$, let $w\in M(\omega_6)_{\omega_6}$ be nonzero. Since $T(\lambda)$ is generated by $w^{\otimes m_6}\in M(\omega_6)^{\otimes m_6}$ one easily checks that $(x_{\beta_{30},1}^-)^{r}w^{\otimes m_6}\ne 0$ for all $r\le m_6$. In particular,
\begin{equation}\label{e:coord6}
(x_{\beta_{30},1}^-)^{r}v\ne 0 \qquad\text{iff}\qquad r\le m_6.
\end{equation}
By the remark closing Subsection \ref{ss:ub}, $T(\lambda)=\sum_{r=0}^{m_6} U(\lie n^-)(x_{\beta_{30},1}^-)^{r}w^{\otimes m_6}$. Hence,  $(x_{\beta_{30},1}^-)^{r}w^{\otimes m_6}$ must be a highest-weight vector in $T(\lambda)[r]$ which implies \eqref{e:lb6}.

\subsubsection{} Next, let $i=2$. The proof is parallel to the previous case. Namely,  \eqref{e:lb} becomes equivalent to
\begin{equation}\label{e:lb2}
t_{(m_2-r)\omega_2+r\omega_5,r}\ne 0 \quad\text{for all}\quad 0\le r\le m_2.
\end{equation}
Notice that $x_{\beta_{24},1}^-$ plays the role that $x_{\beta_{30},1}^-$ did in the case $i=6$. If $m_2=1$, we again use the construction given by \eqref{e:construction} this time with $V_0=V(\omega_2)$ and $V_1=V(\omega_5)$. In particular, it follows that $x_{\beta_{24},1}^-v\ne 0$. For $m_2> 1$, let $w\in M(\omega_2)_{\omega_2}$ be nonzero. As before, we conclude that $(x_{\beta_{24},1}^-)^{r}w^{\otimes m_2}\ne 0$ for all $0\le r\le m_2$. Equation \eqref{e:lb2} follows as in the previous case by using that $T(\lambda)=\sum_{r=0}^{m_2} U(\lie n^-)(x_{\beta_{24},1}^-)^{r}w^{\otimes m_2}$.

We now record the following lemma which, in particular, proves \eqref{e:coord} for $i=2$.

\begin{lem}\label{l:coord2}
Let $r_j\in\mathbb Z_{\ge 0}, j=1,\dots, 5$, and $w=(x_{\beta_{24},1}^-)^{r_1}(x_{\beta_{26},1}^-)^{r_2} (x_{\beta_{28},1}^-)^{r_3} (x_{\beta_{29},1})^{r_4} (x_{\beta_{30},1}^-)^{r_5}v $. Then $w$ is a nonzero scalar multiple of
$$(x_6^-)^{r_5}(x_3^-)^{r_5+r_4}(x_4^-)^{r_5+r_4+r_3}(x_5^-)^{r_5+r_4+r_3+r_2} (x_{\beta_{24},1}^-)^{r_1+r_2+r_3+r_4+r_5}v.$$
Moreover, $w$ is nonzero iff $r_1+\cdots+r_5\le m_2$. In particular, $R(2,m_2,1)=R'(2,m_2,1)$.
\end{lem}

\begin{proof}
The last statement follows immediately from the second. The first statement follows from straightforward successive applications of Lemma \ref{l:heis}. Namely, we first consider the Heisenberg subalgebra generated by $x=x_6^-, y=x_{\beta_{29},1}^-$, and $z=x_{\beta_{30},1}^-$ together with the relation $xv=0$ to get
$$(x_{\beta_{29},1}^-)^{r_4}(x_{\beta_{30},1}^-)^{r_5}v = \eta(x_6^-)^{r_5}(x_{\beta_{29},1}^-)^{r_4+r_5}v$$
for some nonzero scalar $\eta$. Since $[x_6^-,x_{\beta_j,1}^-]=0$ for $j=24,26,28$, it follows that
$$(x_{\beta_{24},1}^-)^{r_1}(x_{\beta_{26},1}^-)^{r_2} (x_{\beta_{28},1}^-)^{r_3} (x_{\beta_{29},1})^{r_4} (x_{\beta_{30},1}^-)^{r_5}v = \eta(x_6^-)^{r_5}(x_{\beta_{24},1}^-)^{r_1}(x_{\beta_{26},1}^-)^{r_2} (x_{\beta_{28},1}^-)^{r_3} (x_{\beta_{29},1})^{r_4+r_5}v.$$
By similarly considering the subalgebras generated by $\{x_3^-,x_{\beta_{28},1}^-, x_{\beta_{29},1}^-\}$, $\{x_4^-,x_{\beta_{26},1}^-, x_{\beta_{28},1}^-\}$, and $\{x_5^-,x_{\beta_{24},1}^-, x_{\beta_{26},1}^-\}$ in this order, the first statement follows.

We have seen above that $(x_{\beta_{24},1}^-)^{r}w^{\otimes m_2}\ne 0$ iff $r\le m_2$. This implies $(x_{\beta_{24},1}^-)^{r_1+r_2+r_3+r_4+r_5}v\ne 0$ iff $r_1+\cdots+r_5\le m_2$. Since $x_5^+(x_{\beta_{24},1}^-)^rv = (x_{\beta_{24},1}^-)^rx_5^+v=0$ and $h_5(x_{\beta_{24},1}^-)^rv = rv$, it follows that $(x_5^-)^s(x_{\beta_{24},1}^-)^rv\ne 0$ for all $0\le s\le r$. In particular, $(x_5^-)^{r_5+r_4+r_3+r_2} (x_{\beta_{24},1}^-)^{r_1+r_2+r_3+r_4+r_5}v\ne 0$. The proof is completed proceeding similarly.
\end{proof}

\subsubsection{} The case $i=4$ is obtained from the previous case by using the nontrivial Dynkin diagram automorphism of $\lie g$. In particular we have:

\begin{lem}\label{l:coord4}
Let $r_j\in\mathbb Z_{\ge 0}, j=1,\dots, 5$, and $w=(x_{\beta_{25},1}^-)^{r_1}(x_{\beta_{27},1}^-)^{r_2} (x_{\beta_{28},1}^-)^{r_3} (x_{\beta_{29},1})^{r_4} (x_{\beta_{30},1}^-)^{r_5}v$. Then $w$ is a nonzero scalar multiple of
$$(x_6^-)^{r_5}(x_3^-)^{r_5+r_4}(x_2^-)^{r_5+r_4+r_3}(x_1^-)^{r_5+r_4+r_3+r_2} (x_{\beta_{25},1}^-)^{r_1+r_2+r_3+r_4+r_5}v.$$
Moreover, $w$ is nonzero iff $r_1+\cdots+r_5\le m_4$. In particular, $R(4,m_4,1)=R'(4,m_4,1)$.\hfill\qedsymbol
\end{lem}

\subsection{Lower bounds}\label{ss:lb} We now complete the proof of \eqref{e:lb} for $\lambda$ as in Theorem \ref{t:main}. In fact, we will carry out most of the proof assuming only that $\lambda(h_3)=0$. Recall the notation $\gbr x_{\gbr r},\gbr r\in\cal A$, developed in Section \ref{ss:ub}. In addition, we shall use the following notation. Denote by $v_{i,m_i}$ a nonzero vector in $M(m_i\omega_i)_{m_i\omega_i}$ and by $v_{i,m_i}^s$ the image of $v_{i,m_i}$ in $M(m_i\omega_i)(s)$. By definition of the truncated module  $M(m_i\omega_i)(s)$ we have
\begin{equation}\label{e:trvanish}
M(m_i\omega_i)(s)[r] =0 \qquad\text{if}\qquad r>s.
\end{equation}
Given $\gb s=(s_i)_{i\in I}\in\mathbb Z_{\ge 0}^I$,  let $T_{\gb s}(\lambda)$ be the submodule of $\otm_{i\in I}^{} M(m_i\omega_i)(s_i)$ generated by $v_{_\gb s}:=\otm_{i\in I}^{} v_{i,m_i}^{s_i}$. Since $T(\lambda)$ is the submodule of $\otm_{i\in I}^{} M(m_i\omega_i)$ generated by $v:= \otm_{i\in I}^{} v_{i,m_i}$,  there exists a unique epimorphism from $T(\lambda)$ onto  $T_{\gb s}(\lambda)$  such that $v\mapsto v_{_\gb s}$.
Let $t_{\mu,r}^{\gb s}$ denote the multiplicity of $V(\mu)$ as an irreducible constituent of $T_\gb s(\lambda)[r]$. Observe that, since $|\cal A_{\wt(\gbr r),\gr(\gbr r)}|= 1$ for all $\gbr r\in\cal A$, in order to prove \eqref{e:lb}, it suffices to prove that
\begin{equation}\label{e:lbtr}
\text{for each}\quad\gbr r\in\cal A\quad\text{there exists}\quad \gb s\in\mathbb Z_{\ge 0}^I\quad\text{such that}\quad t_{\wt(\gbr r),\gr(\gbr r)}^{\gb s}\ge 1.
\end{equation}
It will be convenient to write the tensor product $\otm_{i\in I}^{} M(m_i\omega_i)$ in the following order: $M(m_2\omega_2)\otimes M(m_4\omega_4)\otimes M(m_6\omega_6)\otimes  M(m_1\omega_1)\otimes M(m_5\omega_5)$, where we already used that $m_3=0$ and, hence, $M(m_3\omega_3)\cong V(0)\cong \mathbb C$. In particular, $v=v_{2,m_2}\otimes v_{4,m_4}\otimes v_{6,m_6}\otimes v_{1,m_1}\otimes v_{5,m_5}$ and similarly for $v_{_\gb s},\gb s\in \mathbb Z_{\ge 0}^I$. To shorten notation we write $w= v_{1,m_1}\otimes v_{5,m_5}$  when convenient so that
$$v=v_{2,m_2}\otimes v_{4,m_4}\otimes v_{6,m_6}\otimes w.$$

Let $\{\gbr e_j:j=1,\dots,6\}$ be the canonical basis of $\mathbb Z_{\ge 0}^6$. Given $r\in\mathbb Z$, set $\mathbb Z^6[r]=\{\gbr r\in\mathbb Z^6:\gr(\gbr r)=r\}$, and observe that $\mathbb Z^6[0]$ is a free $\mathbb Z$-module having $\gbr b:=\{(\gb e_1-\gb e_5), (\gb e_2-\gb e_5), (\gb e_5-\gb e_3), (\gb e_5-\gb e_4), (\gb e_5-\gb e_6)\}$ as an ordered $\mathbb Z$-basis. Define $\gbr b_j\in \gbr b, j=1,\dots,5$, by requiring that $\gbr b=\{\gbr b_1,\dots,\gbr b_5\}$ as an ordered set.
Clearly, $\gbr r, \gbr r'\in\mathbb Z^6[r]$ iff $\gbr r-\gbr r'\in\mathbb Z^6[0]$.
Given $\gb j=(j_1,j_2,j_3,j_4,j_5)\in\mathbb Z^5$ and $\gb s\in\mathbb Z_{\ge 0}^I$ such that $s_2\le m_2,s_4\le m_4, s_6\le m_6$, observe that $\gbr r_{_\gbr o}=(s_2,s_4,0,0,0,s_6)\in\cal A_{s_2+s_4+s_6}$ and set
\begin{equation*}
\gbr r_{_\gb j}=\gbr r_{\gbr o} - \sum_{l=1}^5 j_l\gbr b_l = (s_2-j_1,s_4-j_2,j_3,j_4,j_1+j_2-j_3-j_4-j_5,s_6+j_5).
\end{equation*}
Thus, $\gbr r\in\mathbb Z^6[s_2+s_4+s_6]$ iff $\gbr r=\gbr r_{_\gb j}$ for some $\gb j\in\mathbb Z^5$. For shortening some expressions, given $\gb j\in\mathbb Z^5$, we may use the notation $j_0=j_1+j_2-j_3-j_4-j_5$. Notice that $\gbr r_{_\gb j}\in \cal A$  iff
\begin{gather}\notag
\qquad 0\le j_3\le m_5, \qquad 0\le j_4\le m_1, \qquad j_1\le s_2, \qquad j_2\le s_4, \qquad j_0\ge 0,\\ \label{e:B(s)A}\\ \notag
j_5\le m_6-s_6, \qquad j_1-j_3-j_5\le m_4-s_4, \qquad j_2-j_4-j_5\le m_2-s_2.
\end{gather}
Set
$$\cal A(\gb s) = \{\gbr r\in\cal A:  \gbr x_{_\gbr r}v_{_\gb s}\ne 0\}\cap \mathbb Z^6[s_2+s_4+s_6]$$
and let $\cal B(\gb s)$ be the set of tuples  $\gb j\in\mathbb Z_{\ge 0}^5$ satisfying
\begin{gather}\notag
j_3\le j_1\le s_2, \qquad j_4\le j_2\le s_4, \qquad j_3\le m_5, \qquad j_4\le m_1, \qquad j_0\ge 0,\\ \label{e:B(s)}\\ \notag
j_5\le m_6-s_6, \qquad j_1-j_3-j_5\le m_4-s_4, \qquad j_2-j_4-j_5\le m_2-s_2.
\end{gather}
In Subsection \ref{ss:AsBs} we will show that
\begin{equation}\label{e:AsBs}
\gbr r_{_\gb j}\in \cal A(\gb s) \qquad\Leftrightarrow\qquad \gb j\in\cal B(\gb s).
\end{equation}

It follows from \eqref{e:AsBs} that
\begin{equation}\label{e:Ts2s4}
T_{\gb s}(\lambda)[s_2+s_4+s_6] = \sum_{\gb j\in\cal B(\gb s)} U(\lie n^-)\gbr x_{_{\gb r_{_\gb j}}}v_{_\gb s}.
\end{equation}
For $\gb j,\gb k\in\cal B(\gb s)$ we have
\begin{align}\notag
\wt(\gbr r_{_\gb j})-\wt(\gbr r_{_\gb k}) &=\ (k_2-j_2)\alpha_1+(k_1-j_1)\alpha_5+(k_5-j_5)(\alpha_3+\alpha_6)+\\ \label{e:lbwt}\\ \notag
&\ +(k_2-j_2+j_4-k_4)\alpha_2+(k_1-j_1+j_3-k_3)\alpha_4.
\end{align}
In particular, $\wt(\gbr r_{_\gbr o})$ is the unique maximal weight of $T_{\gb s}(\lambda)[s_2+s_4+s_6]$ and, hence,
\begin{equation}\label{e:lbtr0}
t_{\wt(\gbr r_{_\gbr o}),s_2+s_4+s_6}^{\gb s}\ge 1.
\end{equation}

\begin{lem}\label{l:lbtrj50}
Let $\gbr r\in\cal A$. Then, there exists $\gb s\in\mathbb Z_{\ge 0}^I$ and $\gb j\in\cal B(\gb s)$ such that $j_5=0$ and $\gbr r=\gbr r_{_\gb j}$.
In particular, $\gbr r\in\cal A(\gb s)$.
\end{lem}

\begin{proof}
Let $s_1=s_3=s_5=0, s_2=r_1+r_3+r_5, s_4=r_2+r_4$, and $s_6=r_6$. As before, set $\gbr r_{_\gbr o}=(s_2,s_4,0,0,0,s_6)$ and notice that $\gbr r_{_\gbr o}\in\cal A$. One easily checks that $\gbr r=\gbr r_{\gb j}$ where $\gb j=(r_3+r_5,r_4,r_3,r_4,0)$. By \eqref{e:AsBs}, $\gb r\in\cal A(\gb s)$ iff $\gb j\in\cal B(\gb s)$. The checking of the latter is straightforward.
\end{proof}

The above lemma shows that it suffices to show \eqref{e:lbtr} in the case that $\gbr r=\gbr r_{_\gb j}$ for some $\gb s\in\mathbb Z_{\ge 0}^I$ and $\gb j\in\cal B(\gb s)$ such that $j_5=0$. In this case, it follows from the proof of \eqref{e:AsBs} (see the last line of Subsection \ref{ss:AsBs}) that $\gbr x_{_{\gbr r_{_\gb j}}} v_{_\gb s}$ is a nonzero scalar multiple of
\begin{align*}
v_{_\gb j}:=(x^-_{\beta_{28},1})^{j_1-j_3} (x^-_{\beta_{26},1})^{j_3}(x^-_{\beta_{24},1})^{s_2-j_1}v_{2,m_2}^{s_2}\otimes (x^-_{\beta_{28},1})^{j_2-j_4}(x^-_{\beta_{27},1})^{j_4}(x^-_{\beta_{25},1})^{s_4-j_2}v_{4,m_4}^{s_4}\otimes w'
\end{align*}
where $w'=(x^-_{\beta_{30},1})^{s_6}v_{6,m_6}^{s_6}\otimes w$. Notice $v_{_\gb j}\ne 0$ since $\gb j\in\cal B(\gb s)$. From now on we fix $\gb s\in\mathbb Z_{\ge 0}^I$, write $\cal B=\cal B(\gb s)$, and set
$$\cal B_0 = \{\gb j\in\cal B(\gb s): j_5=0\}.$$
Given $\gb k\in\cal B$, let
$$\cal B_{\gb k}^+=\{\gb j\in\cal B(\gb s): \wt(\gbr r_{_\gb k})< \wt(\gbr r_{_\gb j})\} \qquad\text{and}\qquad\cal B_\gb k=\cal B_\gb k^+\cup\{\gb k\}.$$
It easily follows from \eqref{e:lbwt} that
\begin{equation}\label{e:lbwt0}
\gb k\in\cal B_0 \quad\Rightarrow\quad \cal B_\gb k\subseteq \cal B_0.
\end{equation}

By Lemma \ref{l:hwvecs}, \eqref{e:Ts2s4}, and the injectivity of $\wt:\cal A\to P$, \eqref{e:lbtr} holds for $\gbr r=\gbr r_{\gb k}$ iff
\begin{equation}\label{e:lbtrv}
v_{_\gb k} \notin V_\gb k^+:=\sum_{\gb j\in\cal B_{\gb k}^+} U(\lie n^-)v_{_\gb j}.
\end{equation}
Equivalently,  \eqref{e:lbtr} holds for $\gbr r=\gbr r_{\gb k}$ iff  we have an isomorphism of $\lie g$-modules
\begin{equation}\label{e:lb1245}
V_\gb k:=\sum_{\gb j\in\cal B_\gb k} U(\lie n^-) v_{_\gb j} \cong \bigoplus_{\gb j\in\cal B_\gb k}^{} V(\wt(\gbr r_{_\gb j})).
\end{equation}
Given $\gb j\in\cal B_0$, define the height of $\gb j$ to be
\begin{equation*}
\het(\gb j) = \het(\wt(\gbr r_{_\gbr o })-\wt(\gbr r_{_{\gb j}})) = 2(j_1+j_2)-(j_3+j_4)=j_1+j_2+j_0.
\end{equation*}
We prove \eqref{e:lb1245} by induction on $k=\het(\gb k)$. Equation \eqref{e:lbtr0} implies that \eqref{e:lb1245} holds for $k=0$. Thus, assume $k>0$ and, by induction hypothesis, that \eqref{e:lb1245} holds for $\gb j\in\cal B_0$ such that $\het(\gb j)<k$. It follows from the induction hypothesis and \eqref{e:lbwt0} that
\begin{equation}\label{e:lb1245ind}
\dim((V_\gb k^+)_{\wt(\gbr r_{_\gb k})})  = \sum_{\gb j\in\cal B_{\gb k}^+} \dim(V(\wt(\gbr r_{_\gb j}))_{\wt(\gbr r_{_\gb k})}).
\end{equation}
We are left to show that
\begin{equation}\label{e:lb1245indc}
\dim((V_\gb k)_{\wt(\gbr r_{_\gb k})}) =  \dim((V_\gb k^+)_{\wt(\gbr r_{\gb k})})+1.
\end{equation}

Let $J_-=\{1,2\}, J_+=\{4,5\}, J=J_-\cup J_+\subseteq I$ so that $\lie g_{_{J_\pm}}\cong \lie{sl}_3$  and $\lie g_{_{J}}\cong \lie{sl}_3\oplus\lie{sl}_3$. By Proposition \ref{p:restriction}, $U(\lie g_{_{J_\pm }})V(\wt(\gbr r_{_\gb j}))_{\wt(\gbr r_{_\gb j})}\cong V(\wt(\gbr r_{_\gb j})_{J_\pm})$ and similarly for $J$ in place of $J_\pm$. Moreover, we have isomorphisms of vector spaces
\begin{equation}
V(\wt(\gbr r_{_\gb j}))_{\wt(\gbr r_{_\gb k})}\cong V(\wt(\gbr r_{_\gb j})_{J})_{\wt(\gbr r_{_\gb k})_{J}} \cong V(\wt(\gbr r_{_\gb j})_{J_-})_{\wt(\gbr r_{_\gb k})_{J_-}}\otimes V(\wt(\gbr r_{_\gb j})_{J_+})_{\wt(\gbr r_{_\gb k})_{J_+}}.
\end{equation}
The first isomorphism above is clear and the second follows from Proposition \ref{p:sumofalgs}. If $\gb j\in\cal B_\gb k$, it easily follows from \eqref{e:lbwt} that
\begin{gather*}
k_2-j_2\le \wt(\gbr r_{_\gb j})(h_1) = m_1+s_4-j_2-j_4, \qquad k_2-j_2+j_4-k_4\le \wt(\gbr r_{_\gb j})(h_2) = m_2-s_2-j_2+2j_4,\\
k_1-j_1\le \wt(\gbr r_{_\gb j})(h_5) = m_5+s_2-j_1-j_3, \qquad k_1-j_1+j_3-k_3\le \wt(\gbr r_{_\gb j})(h_4) = m_4-s_4-j_1+2j_3,
\end{gather*}
Hence, we can use Lemma \ref{l:sl3} to compute
\begin{gather}\notag
\dim(V(\wt(\gbr r_{_\gb j})_{J_-})_{\wt(\gbr r_{_\gb k})_{J_-}}) = \min\{k_2-j_2,k_2-j_2+j_4-k_4\}+1\\ \text{and}\\ \notag \dim(V(\wt(\gbr r_{_\gb j})_{J_+})_{\wt(\gbr r_{_\gb k})_{J_+}}) = \min\{k_1-j_1,k_1-j_1+j_3-k_3\}+1.
\end{gather}
Plugging this in \eqref{e:lb1245ind} we get
\begin{equation}\label{e:lb1245indim}
\dim((V_\gb k^+)_{\wt(\gbr r_{_\gb k})}) =  \sum_{\gb j\in\cal B_{\gb k}^+} (\min\{k_2-j_2,k_2-j_2+j_4-k_4\}+1) (\min\{k_1-j_1,k_1-j_1+j_3-k_3\}+1).
\end{equation}

We will need the following notation. Given, $i_1,i_2,\dots,i_l\in I$, and $a_1,\dots, a_l\in\mathbb Z_{\ge 0}$, set
\begin{equation*}
\gbr x_{i_1,\dots,i_l}^{a_1,\dots,a_l} = (x_{i_1}^-)^{(a_1)}\cdots (x_{i_l}^-)^{(a_l)}.
\end{equation*}
Also, given $\gb j\in\cal B_{\gb k}$, set
\begin{gather*}
l_-(\gb j) =  \min\{k_2-j_2,k_2-j_2+j_4-k_4\}, \quad l_+(\gb j) =  \min\{k_1-j_1,k_1-j_1+j_3-k_3\}
\end{gather*}
so that \eqref{e:lb1245indc} can be rewritten as
\begin{equation}\label{e:lb1245indc'}
\dim((V_\gb k)_{\wt(\gbr r_{_\gb k})}) =  \sum_{\gb j\in\cal B_{\gb k}} (l_-(\gb j)+1) (l_+(\gb j)+1).
\end{equation}
It now follows from Lemma \ref{l:basis} and \eqref{e:Ts2s4} that \eqref{e:lb1245indc'} holds iff the vectors
\begin{equation}\label{e:1245basis}
\gbr x_{5,4,5}^{p_5,p_4(\gb j),k_1-j_1-p_5}\gbr x_{1,2,1}^{p_1,p_2(\gb j),k_2-j_2-p_1}v_\gb j \qquad\text{are linearly independent}
\end{equation}
for $\gb j\in\cal B_\gb k, 0\le p_1\le l_-(\gb j),0\le p_5\le l_+(\gb j)$. Here $p_2(\gb j)=k_2-j_2+j_4-k_4$ and $p_4(\gb j)=k_1-j_1+j_3-k_3$. We will prove \eqref{e:1245basis} only for $\lambda$ as in Theorem \ref{t:main}. However, let us develop for a little longer the general case. In particular, we will show that all the vectors in \eqref{e:1245basis} are nonzero.

Set $\gb p=(p_1,p_5), \gbr x_{_{\gb j,\gb p}} = \gbr x_{5,4,5}^{p_5,p_4(\gb j),k_1-j_1-p_5}\gbr x_{1,2,1}^{p_1,p_2(\gb j),k_2-j_2-p_1}$, and $v_{_{\gb j,\gb p}}=\gbr x_{_{\gb j,\gb p}}v_{_\gb j}$. Thus, we want to show that the vectors $v_{_{\gb j,\gb p}}$ are linearly independent for $\gb j$ and $\gb p$ as above. From now on, when now confusion arises, we simplify notation and write $l_-$ in place of $l_-(\gb j)$, etc. Recall that
\begin{equation*}
v_{_\gb j}=(x^-_{\beta_{28},1})^{j_1-j_3} (x^-_{\beta_{26},1})^{j_3}(x^-_{\beta_{24},1})^{s_2-j_1}v_{2,m_2}^{s_2}\otimes (x^-_{\beta_{28},1})^{j_2-j_4}(x^-_{\beta_{27},1})^{j_4}(x^-_{\beta_{25},1})^{s_4-j_2}v_{4,m_4}^{s_4}\otimes w'.
\end{equation*}
To simplify the expression above, set $v_6=(x^-_{\beta_{30},1})^{s_6}v_{6,m_6}^{s_6}$, $\gbr x_{_\gb j}^2=(x^-_{\beta_{28},1})^{j_1-j_3} (x^-_{\beta_{26},1})^{j_3}(x^-_{\beta_{24},1})^{s_2-j_1}$, and $\gbr x_{_\gb j}^4 = (x^-_{\beta_{28},1})^{j_2-j_4}(x^-_{\beta_{27},1})^{j_4}(x^-_{\beta_{25},1})^{s_4-j_2}$ so that
\begin{equation}
v_{_\gb j} = \gbr x_{_\gb j}^2v_{2,m_2}^{s_2}\otimes \gbr x_{_\gb j}^4v_{4,m_4}^{s_4}\otimes v_6\otimes w.
\end{equation}
Also, using Lemma \ref{l:coord2} we get
\begin{gather}\label{e:coordli}
\gbr x_{_\gb j}^2v_{2,m_2}^{s_2} = (x_4^-)^{j_1-j_3}(x_5^-)^{j_1}(x^-_{\beta_{24},1})^{s_2}v_{2,m_2}^{s_2} \quad\text{and}\quad \gbr x_{_\gb j}^4v_{4,m_4}^{s_4} = (x_2^-)^{j_2-j_4}(x_1^-)^{j_2}(x^-_{\beta_{25},1})^{s_4}v_{4,m_4}^{s_4}
\end{gather}
up to nonzero scalar multiples. By applying the comultiplication one sees that $v_{_{\gb j,\gb p}}$ is equal to
\begin{align}\label{e:1245licomult}
\sum_\chi \gbr x_{5,4,5}^{d_2,e_2,f_2}\gbr x_{1,2,1}^{a_2,b_2,c_2}\gbr x_{_\gb j}^2v_{2,m_2}^{s_2} \otimes \gbr x_{5,4,5}^{d_4,e_4,f_4}\gbr x_{1,2,1}^{a_4,b_4,c_4}\gbr x_{_\gb j}^4v_{4,m_4}^{s_4}\otimes v_6\otimes \gbr x_{1,2,1}^{a_1,b_1,c_1}v_{1,m_1}\otimes \gbr x_{5,4,5}^{d_5,e_5,f_5}v_{5,m_5}
\end{align}
where $\chi$ runs over the set of collections of nonnegative integers $a_l,b_l,c_l,d_l,e_l,f_l$ satisfying
\begin{gather}\notag
a_2+a_4+a_1=p_1, \qquad b_2+b_4+b_1=p_2, \qquad c_2+c_4+c_1=k_2-j_2-p_1,\\ \label{e:abcdef}\\ \notag
d_2+d_4+d_5=p_5, \qquad e_2+e_4+e_5=p_4, \qquad f_2+f_4+f_5=k_1-j_1-p_5.
\end{gather}
Above we also used that $\gbr x_{1,2,1}^{a,b,c}v_{5,m_5}=\gbr x_{5,4,5}^{a,b,c}v_{1,m_1}=\gbr x_{1,2,1}^{a,b,c}v_6=\gbr x_{5,4,5}^{a,b,c}v_6=0$ whenever $a+b+c>0$. We will need to study the summands on the right-hand-side of \eqref{e:1245licomult}.

Using Lemma \ref{l:sl3heis} we see that $\gbr x_{1,2,1}^{a_1,b_1,c_1}v_{1,m_1}\ne 0$ iff $a_1+c_1\le m_1$ and $b_1\le c_1$ and, in that case, $\gbr x_{1,2,1}^{a_1,b_1,c_1}v_{1,m_1}=\eta\gbr x_{2,1}^{b_1,a_1+c_1}v_{1,m_1}$ for some positive rational number $\eta$ (depending on $a_1,b_1,c_1$). Similarly, $\gbr x_{5,4,5}^{d_5,e_5,f_5}v_{5,m_5}\ne 0$ iff $d_5+f_5\le m_5$ and $e_5\le f_5$ and, in that case, $\gbr x_{5,4,5}^{d_5,e_5,f_5}v_{5,m_5}$ is a positive multiple of $\gbr x_{4,5}^{e_5,d_5+f_5}v_{5,m_5}$.
Next, we study the factor $\gbr x_{5,4,5}^{d_2,e_2,f_2}\gbr x_{1,2,1}^{a_2,b_2,c_2}\gbr x_{_\gb j}^2v_{2,m_2}^{s_2}=\gbr x_{1,2,1}^{a_2,b_2,c_2}\gbr x_{5,4,5}^{d_2,e_2,f_2}\gbr x_{_\gb j}^2v_{2,m_2}^{s_2}$. Notice that $x_5^+(x^-_{\beta_{24},1})^{s_2}v_{2,m_2}^{s_2}=h_4(x^-_{\beta_{24},1})^{s_2}v_{2,m_2}^{s_2}=0$, and $h_5(x^-_{\beta_{24},1})^{s_2}v_{2,m_2}^{s_2} = s_2$. Therefore, we can use Lemma \ref{l:sl3heis} together with \eqref{e:coordli} to see that $\gbr x_{5,4,5}^{d_2,e_2,f_2}\gbr x_{_\gb j}^2v_{2,m_2}^{s_2}$ is a nonnegative rational multiple of
\begin{equation*}
\gbr x_{4,5}^{e_2+j_1-j_3,j_1+f_2+d_2}(x_{\beta_{24},1}^-)^{s_2}v_{2,m_2}^{s_2}
\end{equation*}
and it is nonzero provided $e_2\le j_3+f_2$ and $j_1+d_2+f_2\le s_2$.  Since $d_2\le p_5, f_2\le k_1-j_1-p_5$ by \eqref{e:abcdef}, and $k_1\le s_2$, the latter is always satisfied. One easily checks that
$$x_1^+\gbr x_{_\gb j}^2v_{2,m_2}^{s_2}=h_1\gbr x_{_\gb j}^2v_{2,m_2}^{s_2}=0$$
which implies $\gbr x_{5,4,5}^{d_2,e_2,f_2}\gbr x_{1,2,1}^{a_2,b_2,c_2}\gbr x_{_\gb j}^2v_{2,m_2}^{s_2}=0$ if $c_2\ne 0$.
Next, using the relations
\begin{gather*}
[x_2^+, x_{\beta_{24},1}^-]=x_{\beta_{21},1}^-, \quad [x_{\beta_{21},1}^-, x_{\beta_{24},1}^-]=0, \quad x_{\beta_{21},1}^-v_{2,m_2}^{s_2}=0,
\end{gather*}
one sees that $x_2^+\gbr x_{_\gb j}^2v_{2,m_2}^{s_2}=0$. Since $h_2\gbr x_{_\gb j}^2v_{2,m_2}^{s_2}=(m_2-s_2)\gbr x_{_\gb j}^2v_{2,m_2}^{s_2}$, it follows from Lemma \ref{l:sl3heis} that
\begin{equation*}
\gbr x_{1,2,1}^{a_2,b_2,c_2}\gbr x_{_\gb j}^2v_{2,m_2}^{s_2}\ne 0 \quad\text{iff}\quad c_2=0 \quad\text{and}\quad a_2\le b_2\le m_2-s_2.
\end{equation*}
Since we anyway have $b_2\le p_2=(k_2-k_4)-(j_2-j_4)\le m_2-s_2$, the relevant conditions are $c_2=0$ and $a_2\le b_2$.
Therefore, we find that $\gbr x_{5,4,5}^{d_2,e_2,f_2}\gbr x_{1,2,1}^{a_2,b_2,c_2}\gbr x_{_\gb j}^2v_{2,m_2}^{s_2}$ is a nonnegative rational multiple of
$$\gbr x_{4,5}^{e_2+j_1-j_3,j_1+f_2+d_2}\gbr x_{1,2}^{a_2,b_2}(x_{\beta_{24},1}^-)^{s_2}v_{2,m_2}^{s_2}$$
which is nonzero iff
$$a_2\le b_2 \qquad\text{and}\qquad e_2\le j_3+f_2.$$
Similarly, we get that $\gbr x_{5,4,5}^{d_4,e_4,f_4}\gbr x_{1,2,1}^{a_4,b_4,c_4}\gbr x_{_\gb j}^4v_{4,m_4}^{s_4}$  is a nonnegative rational multiple of
$$\gbr x_{2,1}^{b_4+j_2-j_4,j_2+c_4+a_4}\gbr x_{5,4}^{d_4,e_4}(x_{\beta_{25},1}^-)^{s_4}v_{4,m_4}^{s_4}$$
which is nonzero iff
$$d_4\le e_4 \qquad\text{and}\qquad b_4\le j_4+c_4.$$
Therefore, the sum in \eqref{e:1245licomult} is a linear combination of the vectors
\begin{equation}\label{e:1245licomult'}
\gbr x_{4,5}^{e'_2,f'_2}\gbr x_{1,2}^{a_2,b_2}(x_{\beta_{24},1}^-)^{s_2}v_{2,m_2}^{s_2} \otimes \gbr x_{5,4}^{d_4,e_4}\gbr x_{2,1}^{b'_4,c'_4}(x_{\beta_{25},1}^-)^{s_4}v_{4,m_4}^{s_4}\otimes v_6\otimes \gbr x_{2,1}^{b_1,c'_1}v_{1,m_1}\otimes \gbr x_{4,5}^{e_5,f'_5}v_{5,m_5}
\end{equation}
where
\begin{gather*}
c_1'=c_1+a_1, \qquad b_4'=b_4+j_2-j_4, \qquad c_4'=c_4+a_4+j_2,\\
f_5'=f_5+d_5, \qquad e_2'=e_2+j_1-j_3, \qquad f_2'=f_2+d_2+j_1,
\end{gather*}
with the numbers $a_l,b_l,\dots,f_l$ satisfying \eqref{e:abcdef} as well as
\begin{gather}\notag
a_1+c_1\le m_1, \qquad b_1\le c_1, \qquad a_2\le b_2,  \qquad b_4\le c_4+j_4, \qquad c_2=0,\\\label{e:abcdef<}\\\notag
d_5+f_5\le m_5, \qquad e_5\le f_5, \qquad d_4\le e_4,  \qquad e_2\le f_2+j_3, \qquad f_4=0.
\end{gather}
Notice that
\begin{gather*}
a_2=a_1=b_1=b_4=c_1=0, \qquad a_4=p_1, \qquad b_2=p_2, \qquad c_4=k_2-j_2-p_1,\\
d_4=d_5=e_5=e_2=f_5=0, \qquad d_2=p_5, \qquad e_4=p_4, \qquad f_2=k_1-j_1-p_5,
\end{gather*}
satisfy \eqref{e:abcdef} and \eqref{e:abcdef<}, which implies that the set of nonzero summands in \eqref{e:1245licomult} is nonempty. One easily sees that the vectors in \eqref{e:1245licomult'},  for distinct values of $(a_2,b_1,b_2,b_4',c_1',c_4',d_4,e_2',e_4,e_5,f_2',f_5')$,  are linearly independent by looking at the weights of their tensor factors. Since $v_{_{\gb j,\gb p}}$ is a linear combination of these vectors with positive rational coefficients, it follows that $v_{_{\gb j,\gb p}}\ne 0$ for all choices of $\gb j$ and $\gb p$.

We now restrict ourselves to $\lambda$ as in Theorem \ref{t:main}. To simplify notation, we rewrite the vectors in \eqref{e:1245licomult'} as
\begin{equation}\label{e:1245licomult's}
v_2^{a_2,b_2,e'_2,f'_2}\otimes v_4^{b'_4,c'_4,d_4,e_4}\otimes v_6\otimes v_1^{b_1,c'_1}\otimes v_5^{e_5,f'_5}.
\end{equation}
If $\{2,4\}\nsubseteq\supp(\lambda)$, the argument reduces to one identical to the one used in the proof of \cite[Proposition 5.7]{mou:reslim} (all the details can be found in \cite[Lemma 5.3.9]{per}). From now on we assume $\supp(\lambda)\subseteq\{2,4,6\}$ which is the remaining case to consider.  In this case, we must have $j_3=j_4=k_3=k_4=0, p_2=l_-, p_4=l_+$. In particular, \eqref{e:abcdef} and \eqref{e:abcdef<} reduce to
\begin{gather*}
a_1=b_1=c_1=c_2=0, \quad a_2+a_4=p_1, \quad b_2+b_4=k_2-j_2, \quad c_4=k_2-j_2-p_1, \quad a_2\le b_2, \quad b_4\le c_4,\\
d_5=e_5=f_5=f_4=0, \quad d_4+d_2=p_5, \quad e_4+e_2=k_1-j_1, \quad f_2=k_1-j_1-p_5, \quad d_4\le e_4, \quad e_2\le f_2.
\end{gather*}
Therefore, $v_{_{\gb j,\gb p}}$ is a linear combination of vectors of the form
\begin{equation}\label{e:24licomult's}
v_2^{a_2,b_2,k_1-e_4,k_1-d_4}\otimes v_4^{k_2-b_2,k_2-a_2,d_4,e_4}\otimes v_6 \quad\text{with}\quad 0\le a_2\le p_1\le b_2\le l_-, 0\le d_4\le p_5\le e_4\le l_+.
\end{equation}
Set
\begin{equation}
v_{a,b,d,e} = v_2^{a,b,k_1-e,k_1-d}\otimes v_4^{k_2-b,k_2-a,d,e}\otimes v_6
\end{equation}
and observe that the coefficient of $v_{a,b,d,e}$ in $v_{_{\gb j,\gb p}}$ is nonzero iff $j_1\le k_1-e, j_2\le k_2-b, a\le p_1, d\le p_5$.

To complete the proof, we now show by induction on $n_1\in\mathbb Z_{\ge 0}$ that the set $\{v_{_{\gb j,\gb p}}: (k_1-j_1)\le n_1\}$ is linearly independent. We prove this performing a further induction on $n_2\in\mathbb Z_{\ge 0}$ to show that the set $\{v_{_{\gb j,\gb p}}: (k_1-j_1)\le n_1, (k_2-j_2)\le n_2\}$ is linearly independent. Set
\begin{gather*}
S(n_1,n_2)=\{(\gb j,\gb p): k_1-j_1\le n_1, k_2-j_2\le n_2\}, \quad S[n_1,n_2) = \{(\gb j,\gb p): k_1-j_1= n_1, k_2-j_2\le n_2\},\\
S(n_1,n_2] = \{(\gb j,\gb p): k_1-j_1\le n_1, k_2-j_2= n_2\}, \quad S[n_1,n_2] = \{(\gb j,\gb p): k_1-j_1= n_1, k_2-j_2= n_2\}.
\end{gather*}
The inductions clearly start when $n_1=n_2=0$ since  $\{v_{_{\gb j,\gb p}}:(\gb j,\gb p)\in S(0,0)\}=\{v_{_\gb k}\}$. Assume now that $n_2>0$ and, by induction hypothesis, that the set  $\{v_{_{\gb j,\gb p}}: (\gb j,\gb p)\in S(n_1,n_2-1)\}$ is linearly independent.  Let $c_{_{\gb j,\gb p}}\in\mathbb C$ be such that
\begin{equation}
\sum_{(\gb j,\gb p)\in S(n_1,n_2)} c_{_{\gb j,\gb p}}v_{_{\gb j,\gb p}} = 0.
\end{equation}
By the induction hypothesis, it remains to show that
\begin{equation}\label{e:cjp=0}
c_{_{\gb j,\gb p}}=0  \quad\text{for all}\quad (\gb j,\gb p)\in S(n_1,n_2].
\end{equation}
Set
\begin{gather*}
S[n_1,n_2](m) = \{(\gb j,\gb p)\in S[n_1,n_2]: (p_1,p_5) = (n_2-r,n_1-s), r+s\le m\}.
\end{gather*}
Observe that if $(\gb j,\gb p)\in S(n_1,n_2)$ is such that the coefficient of $v_{n_2-r,n_2,n_1-s,n_1}$ in  $v_{_{\gb j,\gb p}}$ is nonzero, then $(\gb j,\gb p)\in S[n_1,n_2]$ and $(p_1,p_5) = (n_2-r',n_1-s'), 0\le r'\le r, 0\le s'\le s$. An easy induction on $r+s\ge 0$ shows that $c_{_{\gb j,\gb p}}=0$ for all $(\gb j,\gb p)\in S[n_1,n_2](r+s)$. This implies $c_{_{\gb j,\gb p}}=0$ for all $(\gb j,\gb p)\in S[n_1,n_2]$. Similarly, if $(\gb j,\gb p)\in S(n_1,n_2)\backslash S[n_1,n_2]$ is such that the coefficient of $v_{n_2-r,n_2,n_1-1-s,n_1-1}$ in  $v_{_{\gb j,\gb p}}$ is nonzero, then $(\gb j,\gb p)\in S[n_1-1,n_2]$ and $(p_1,p_5) = (n_2-r',n_1-1-s'), 0\le r'\le r, 0\le s'\le s$. Again, an easy induction on $r+s\ge 0$ shows that $c_{_{\gb j,\gb p}}=0$ for all $(\gb j,\gb p)\in S[n_1-1,n_2](r+s)$. Proceeding recursively in this way one proves $c_{_{\gb j,\gb p}}=0$ for all $(\gb j,\gb p)\in S[n_1-j,n_2], 0\le j\le n_1$. Since $S(n_1,n_2] = \cup_j S[n_1-j,n_2]$, \eqref{e:cjp=0} follows.

The above paragraph proves the induction step on $n_2$. It remains to show that the induction on $n_2$ starts when $n_1>0$. Thus, assume $n_1>0, n_2=0$ and, by induction hypothesis on $n_1$, that $\{v_{_{\gb j,\gb p}}:(\gb j,\gb p)\in S(n_1-1,0)\}$ is linearly independent. Let $c_{_{\gb j,\gb p}}\in\mathbb C$ be such that
\begin{equation}
\sum_{(\gb j,\gb p)\in S(n_1,0)} c_{_{\gb j,\gb p}}v_{_{\gb j,\gb p}} = 0.
\end{equation}
By the induction hypothesis, it remains to show that
\begin{equation}\label{e:cjp=0n2=0}
c_{_{\gb j,\gb p}}=0  \quad\text{for all}\quad (\gb j,\gb p)\in S[n_1,0].
\end{equation}
The proof of \eqref{e:cjp=0n2=0} is similar to that of \eqref{e:cjp=0} and we omit the details.

\begin{rem}
Observe that the above proof of \eqref{e:cjp=0} is based on finding values of $a,b,d,e$ such that $v_{a,b,d,e}$ appears with nonzero coefficient in $v_{_{\gb j, \gb p}}$ for exactly one value of of the pair $(\gb j, \gb p)\in S(n_1,n_2)$ and so on.  The difficult in adapting the above proof for proving \eqref{e:1245basis} for all $\lambda$ not supported in the trivalent node resides in the fact that, if $\{2,4\}\subseteq\supp(\lambda)$ and either $m_1\ne 0$ or $m_5\ne 0$, one can give examples of $(\gb j, \gb p)\ne (\gb j', \gb p')$ such that the summands of the form $v_{a,b,d,e}$ with nonzero coefficients appearing in $v_{_{\gb j, \gb p}}$ are exactly the same as those appearing in $v_{_{\gb j', \gb p'}}$. Hence, one would need to keep a very efficient control of the coefficients.
\end{rem}

\subsection{Proof of \eqref{e:AsBs}}\label{ss:AsBs}
By \eqref{e:B(s)A}, in order to prove that $\gbr r_{_\gb j}\in \cal A(\gb s)\Rightarrow \gb j\in\cal B(s)$, it remains to show that $\gbr x_{_{\gbr r_{_\gb j}}} v_{_\gb s}\ne 0$ only if $j_3\le j_1, j_4\le j_2$, and $j_5\ge 0$. It follows from Lemma \ref{l:coord} that
\begin{align*}
\gbr x_{_{\gbr r_{_\gb j}}} v_{_\gb s} = (x^-_{\beta_{30},1})^{s_6+j_5}(x^-_{\beta_{28},1})^{j_0} \left((x^-_{\beta_{26},1})^{j_3}(x^-_{\beta_{24},1})^{s_2-j_1}v_{2,m_2}^{s_2}\otimes (x^-_{\beta_{27},1})^{j_4}(x^-_{\beta_{25},1})^{s_4-j_2}v_{4,m_4}^{s_4}\otimes v_{6,m_6}^{s_6}\otimes w\right)
\end{align*}
Notice that if $s_2-j_1+j_3>s_2$ we have $(x^-_{\beta_{26},1})^{j_3}(x^-_{\beta_{24},1})^{s_2-j_1}v_{2,m_2}^{s_2}=0$ by \eqref{e:trvanish}. In other words, $\gbr r_{_\gb j}\in\cal A(\gb s)$ only if $j_3\le j_1$. Similarly, we must have $j_4\le j_2$. Continuing the above computation we get that $\gbr x_{_{\gbr r_{_\gb j}}} v_{_\gb s} = (x^-_{\beta_{30},1})^{s_6+j_5}v'$ where $v'$ is the vector
\begin{align*}
&\sum_{k=0}^{j_0}\tbinom{j_0}{k} (x^-_{\beta_{28},1})^{j_0-k} (x^-_{\beta_{26},1})^{j_3}(x^-_{\beta_{24},1})^{s_2-j_1}v_{2,m_2}^{s_2}\otimes (x^-_{\beta_{28},1})^{k}(x^-_{\beta_{27},1})^{j_4}(x^-_{\beta_{25},1})^{s_4-j_2}v_{4,m_4}^{s_4}\otimes v_{6,m_6}^{s_6}\otimes w.
\end{align*}
By \eqref{e:trvanish}, $(x^-_{\beta_{28},1})^{j_0-k} (x^-_{\beta_{26},1})^{j_3}(x^-_{\beta_{24},1})^{s_2-j_1}v_{2,m_2}^{s_2}=0$ if $(j_0-k)+j_3+(s_2-j_1)>s_2$. Hence, the summand corresponding to $k$ in the above summation is nonzero only if $j_2-j_4-j_5\le k$. Similarly, $(x^-_{\beta_{28},1})^{k}(x^-_{\beta_{27},1})^{j_4}(x^-_{\beta_{25},1})^{s_4-j_2}v_{4,m_4}^{s_4}=0$ if $k+j_4+(s_4-j_2)>s_4$, i.e., if $k>j_2-j_4$. Thus, the summand corresponding to $k$ in the above summation is nonzero only if $j_2-j_4-j_5\le k\le j_2-j_4$. In particular, we must have $j_5\ge 0$.

To complete the proof of \eqref{e:AsBs}, we need to show that $\gb j\in\cal B(\gb s)\Rightarrow \gbr x_{_{\gbr r_{_\gb j}}} v_{_\gb s}\ne 0$. Set $j_-=\max\{0,j_2-j_4-j_5\}$ and $j_+=\min\{j_0,j_2-j_4\}$ and observe that $\gb j\in\cal B(\gb s)\Rightarrow j_-\le j_+$. Given $j_-\le k\le j_+$, set
$$v_k = \tbinom{j_0}{k}(x^-_{\beta_{28},1})^{j_0-k} (x^-_{\beta_{26},1})^{j_3}(x^-_{\beta_{24},1})^{s_2-j_1}v_{2,m_2}^{s_2}\otimes (x^-_{\beta_{28},1})^{k}(x^-_{\beta_{27},1})^{j_4}(x^-_{\beta_{25},1})^{s_4-j_2}v_{4,m_4}^{s_4}.$$
Notice that Lemmas \ref{l:coord2} and \ref{l:coord4} imply that $v_k\ne 0$. Continuing the above computation we see that
\begin{align*}
\gbr x_{_{\gbr r_{_\gb j}}} v_{_\gb s}&=\ \sum_{l=0}^{s_6+j_5}\sum_{k=j_-}^{j_+} \tbinom{s_6+j_5}{l}(x^-_{\beta_{30},1})^{s_6+j_5-l}v_k\otimes (x^-_{\beta_{30},1})^{l}v_{6,m_6}^{s_6}\otimes w=\\
&=\ \tbinom{s_6+j_5}{s_6}\sum_{k=j_-}^{j_+} (x^-_{\beta_{30},1})^{j_5}v_k\otimes (x^-_{\beta_{30},1})^{s_6}v_{6,m_6}^{s_6}\otimes w.
\end{align*}
The second equality above is proved as follows. By \eqref{e:trvanish}, $(x^-_{\beta_{30},1})^{s_6+j_5-l}v_k=0$ if $(s_6+j_5-l)+(j_0-k)+j_3+(s_2-j_1)+k+j_4+(s_4-j_2)>s_2+s_4$, i.e., if $l<s_6$. Similarly, $(x^-_{\beta_{30},1})^{l}v_{6,m_6}^{s_6}=0$ if $l>s_6$.
By \eqref{e:coord6}, $(x^-_{\beta_{30},1})^{s_6}v_{6,m_6}^{s_6}\ne 0$ and, therefore, it remains to show that
\begin{equation}\label{e:AsBs'}
(x^-_{\beta_{30},1})^{j_5}\sum_{k=j_-}^{j_+} v_k\ne 0.
\end{equation}
Indeed, $\tbinom{j_0}{k}^{-1}(x^-_{\beta_{30},1})^{j_5}v_k$ is equal to
\begin{align*}
&\sum_{l=0}^{j_5}\tbinom{j_5}{l} (x^-_{\beta_{30},1})^{l}(x^-_{\beta_{28},1})^{j_0-k} (x^-_{\beta_{26},1})^{j_3}(x^-_{\beta_{24},1})^{s_2-j_1}v_{2,m_2}^{s_2}\otimes (x^-_{\beta_{30},1})^{j_5-l} (x^-_{\beta_{28},1})^{k}(x^-_{\beta_{27},1})^{j_4}(x^-_{\beta_{25},1})^{s_4-j_2}v_{4,m_4}^{s_4}.
\end{align*}
Making use of \eqref{e:trvanish} once more we see that
$$(x^-_{\beta_{30},1})^{j_5}v_k=\tbinom{j_0}{k}\tbinom{j_5}{k-j_2+j_4+j_5} v_2^k\otimes v_4^k$$
where
$$v_2^k=(x^-_{\beta_{30},1})^{k-j_2+j_4+j_5}(x^-_{\beta_{28},1})^{j_0-k} (x^-_{\beta_{26},1})^{j_3}(x^-_{\beta_{24},1})^{s_2-j_1}v_{2,m_2}^{s_2}$$
and
$$v_4^k = (x^-_{\beta_{30},1})^{j_2-j_4-k} (x^-_{\beta_{28},1})^{k}(x^-_{\beta_{27},1})^{j_4}(x^-_{\beta_{25},1})^{s_4-j_2}v_{4,m_4}^{s_4}.$$
Lemma \ref{l:coord2} implies that $v_2^k\ne 0$ while Lemma \ref{l:coord4} implies that $v_4^k\ne 0$. Observing that  $v_2^k$ are weight vectors of distinct weight and similarly for $v_4^k$, \eqref{e:AsBs'} follows. This completes the proof of \eqref{e:AsBs}. Notice also that, if $j_5=0$, it follows from the computations above that $\gbr x_{_{\gbr r_{_\gb j}}} v_{_\gb s}$ is a nonzero scalar multiple of
\begin{align*}
(x^-_{\beta_{28},1})^{j_1-j_3} (x^-_{\beta_{26},1})^{j_3}(x^-_{\beta_{24},1})^{s_2-j_1}v_{2,m_2}^{s_2}\otimes (x^-_{\beta_{28},1})^{j_2-j_4}(x^-_{\beta_{27},1})^{j_4}(x^-_{\beta_{25},1})^{s_4-j_2}v_{4,m_4}^{s_4}\otimes (x^-_{\beta_{30},1})^{s_6}v_{6,m_6}^{s_6}\otimes w.
\end{align*}

\bibliographystyle{amsplain}

\end{document}